\documentclass[12pt]{amsart}

\usepackage[utf8]{inputenc}
\usepackage{lmodern} 
\usepackage{microtype} 
\usepackage[a4paper, margin=1in]{geometry}
\usepackage{parskip} 
\numberwithin{equation}{section}

\usepackage{amsmath, amssymb, amsthm}
\usepackage{mathtools} 
\usepackage{mathrsfs}  
\usepackage{bm}        

\theoremstyle{plain}
\newtheorem{theorem}{Theorem}[section]
\newtheorem{lemma}[theorem]{Lemma}
\newtheorem{proposition}[theorem]{Proposition}
\newtheorem{corollary}[theorem]{Corollary}

\theoremstyle{definition}
\newtheorem{definition}[theorem]{Definition}

\theoremstyle{remark}
\newtheorem{remark}[theorem]{Remark}

\newcommand{\R}{\mathbb{R}}

\newcommand{\E}{\mathbb{E}}
\newcommand{\ra}{\rightarrow}

\newcommand{\C}{\mathbb{C}}

\newcommand{\cE}{\mathcal{E}}

\newcommand{\lip}{\langle}
\newcommand{\rip}{\rangle}

\usepackage{enumerate}
\usepackage{amsmath, amssymb}
\usepackage{graphicx}
\usepackage{mathrsfs}
\usepackage{color}

\usepackage[normalem]{ulem}
\usepackage{blkarray}
\usepackage{nicefrac}
\usepackage{ytableau}

\newcommand{\tr}{{\text{Tr}}}

\newcommand{\GL}{\mathrm{GL}}

\usepackage[colorlinks=true, linkcolor=blue, citecolor=blue, urlcolor=blue]{hyperref}
\usepackage[capitalize, nameinlink]{cleveref} 

\usepackage[backend=bibtex,style=alphabetic,sorting=nyt,url=false,doi=false,isbn=false]{biblatex}
\renewbibmacro{in:}{}

\usepackage{hyperref}
\usepackage{tikz}
\usetikzlibrary{arrows.meta, calc, decorations.markings, math}

\title{Free versions of the Strong Szeg\H{o} Limit Theorem}
\author{Michael T. Jury and Lodewyk J. van Rensburg and George Roman}
\address{DEPARTMENT OF MATHEMATICS, UNIVERSITY OF FLORIDA, GAINESVILLE, FL 32611-8105}
\email{mjury@ufl.edu} 
\email{lodewyk.jansenva@ufl.edu}
\email{g.roman@ufl.edu} 
\thanks{MJ partially supported by National Science Foundation Grant DMS-2154494.}
\date{\today}
\begin{document}

\begin{abstract} 
The Strong Szeg\H{o} Limit Theorem is a theorem about the asymptotics of the determinants of large Toeplitz matrices. It can be reformulated as a probabilistic statement about eigenvalue statistics of random unitary matrices. We prove a multivariate generalization of the theorem in this latter form, replacing a single unitary with a system of independent random unitaries. It turns out that the standard proofs of the classical theorem do not generalize to this setting, and instead we must import tools from random matrix theory, free probability, and noncommutative function theory. In addition, we obtain results about the averaged determinants of random unitary pencils, as the size of the unitaries tends to infinity; and an auxiliary result giving a formula for the spectral radius of a matricial sum of free Haar unitaries. 
\end{abstract}
\maketitle 

\section{Introduction}

The purpose of this paper is to formulate and prove certain ``free" multivariable versions of the Strong Szeg\H{o} Limit Theorem (SSLT), in its probabilistic form. The classical SSLT \cite{Szego1952} is a theorem about the asymptotic behavior of the determinants of large Toeplitz matrices, which can be recast as a statement about eigenvalues of functions of large random unitary matrices. Broadly, our goal will be to prove versions of these statements that hold for functions of several independent unitaries. To be able to state our results precisely, we first describe the probabilistic point of view on the classical theorem. 

If $(c_n)_{n\in\mathbb Z}$ is the sequence of Fourier coefficients of an $L^1$ function on the unit circle $\mathbb T$, one can form the Toeplitz matrix 
\[
T_N(f)=
\begin{bmatrix}
c_0      & c_{-1}   & c_{-2}   & \cdots & c_{-(N-1)} \\
c_{1}   & c_{0}   & c_{-1}   & \cdots & c_{-(N-2)} \\
c_{2}   & c_{1}  & c_{0}   & \cdots & c_{-(N-3)} \\
\vdots   & \vdots  & \vdots  & \ddots & \vdots \\
c_{N-1} & c_{N-2} & c_{N-3} & \cdots & c_0
\end{bmatrix}.
\]
We put $D_N(f)=\det T_N(f)$. Then the simplest form of the classical SSLT is
\begin{theorem}[Classical SSLT, determinant form] Let $\phi$ be a sufficiently smooth\footnote{We will not be concerned with the optimal form of the classical theorem in this paper, since we will need relatively strong smoothness assumptions for our ``free" version; in fact in the classical version, for real $\phi$,  it suffices to assume only $\phi,\; \exp(\phi)\in L^1$. See \cite[Chapter 6]{siomn-opuc} and \cite[Chapter 10]{bottcher-silbermann-book} for precise statements and a history of the various proofs.} continuous function on the unit circle. Let $(c_n)_{n\in \mathbb Z}$ be the Fourier coefficients of $\phi$, and suppose that $\sum_{n\in\mathbb Z} |c_n|^2|n|<\infty$. Then
\[
D_N(e^\phi)\sim E_0^N \cdot E_1
\]
where
\[
E_0 =\exp(c_0), \quad E_1= \exp\left(\frac12 \sum_{n\in\mathbb Z} c_{n}c_{-n} |n|\right).
\]
\end{theorem}
(For positive $a_N, b_N$, we write $a_N\sim b_N$ to mean $\lim \frac{a_N}{b_N}=1$.) The SSLT can be recast in a ``probabilistic" form which makes no mention of Toeplitz matrices, but instead involves averages of functions of a random unitary matrix. This proceeds through the so-called ``Coulomb gas representation" of the Toeplitz determinant. By this we mean the following formula: 
\begin{theorem}[Coulomb gas representation]
    Let $\phi, e^\phi\in L^1(\mathbb T)$. Then 
    \[
    D_N(e^\phi) =\frac{1}{N!} \frac{1}{(2\pi)^{N}}\int_{\mathbb T^N} \exp\left( \sum_{j=1}^N \phi (\theta_j)\right)\left|\prod_{1\leq j<k\leq N} (e^{i\theta_k}-e^{i\theta_j}) \right|^2\, d\theta_1\cdots d\theta_N .
    \]
\end{theorem}
The integral on the right can be interpreted as an instance of the Weyl integration formula for integration with respect to Haar measure on the unitary group $\mathcal U(N)$. The Weyl integration formula says that if $\Phi:\mathcal U(N)\to \mathbb C$ is an integrable function which depends only on the eigenvalues of $U$, then its integral (expectation) with respect to Haar measure is 
\begin{equation}\label{eqn:weyl-integration}
\mathbb E_{N} [\Phi(U)] = \frac{1}{N!} \frac{1}{(2\pi)^{N}}\int_{\mathbb T^{N}} \Phi(\theta_1, \dots, \theta_N)\left|\prod_{1\leq j<k\leq N} (e^{i\theta_k}-e^{i\theta_j}) \right|^2\, d\theta_1\cdots d\theta_N .
\end{equation}
In particular, any  continuous function $\phi$ on the unit circle induces a function on unitary matrices via the trace $U\to \tr \phi(U)$, where $\phi(U)$ is defined by the spectral theorem; concretely this function is the sum $\sum \phi(\lambda_j)$ obtained by evaluating $\phi$ at the eigenvalues $\lambda_j=e^{i\theta_j}$ of the unitary matrix $U$. If we apply Weyl integration formula \ref{eqn:weyl-integration} to the function $\Phi(U)= \exp[\tr\, \phi(U)]$, we see that the Coulomb gas representation can be re-expressed in the form
\[
D_N(e^\phi) = \mathbb E_{N}[\exp (\tr \,\phi(U))].
\]
Thus we have
\begin{theorem}[Classical SSLT, probabilistic form]\label{thm:classic-sslt-prob} Let $\phi$ be a sufficiently smooth continuous function on the circle whose Fourier coefficients satisfy $\sum_{n\in\mathbb Z}|c_n|^2|n|<\infty.$ Then 
\begin{equation}\label{eqn:classical-sslt-prob}
\mathbb E_N[\exp( \tr\, \phi(U))]\sim E_0^N\cdot E_1 
\end{equation}
with $E_0$ and $E_1$ as above.
\end{theorem}
The theorem, in this probabilistic form, has an elegant proof due to Bump and Diaconis \cite{bump-diaconis-2002} (under the additional smoothness assumption $\sum_{n\in\mathbb Z}|c_n|<\infty$), which exploits the representation theory of the unitary group to give a combinatorial, rather than analytic, expression for the expectation. From this the asymptotic behavior can be deduced. 

Pushing a little further, for real $\phi$ let $F_N$ denote the random variable $\tr\, \phi(U^{})$ (at unitaries of size $N$). We then have $\mathbb E_N[F_N]= Nc_0$, and if we replace the function $\phi$ by $t\phi$ for a real number $t$ (so that $c_n$ is replaced by $tc_n$), we deduce from \eqref{eqn:classical-sslt-prob} that
\begin{equation}\label{eqn:classical-fluctuation}
\lim_{N\to \infty} \left(\log \mathbb E_N[\exp tF_N] -t\mathbb E_N F_N\right) =\frac12 \left(\sum_{n\in\mathbb Z} |c_n|^2 |n| \right)t^2
\end{equation}
for all real $t$. By a standard application of the method of moments in probability theory, this implies that the recentered random variables $F_N-\mathbb E_NF_N$ converge in distribution to a Gaussian random variable with mean 0 and variance $\sum_{n\in\mathbb Z} |c_n|^2 |n|$. In this way, the SSLT implies a ``fluctuation theorem" for the random variables $F_N=\tr \phi(U)$. (The functions $t\to \mathbb E \exp(tF_N)$, $t\to \log \mathbb E  \exp(tF_N)$ are the {\em moment} and {\em cumulant generating functions} of $F_N$, respectively.)

Before we state our multivariate version of the theorem, it will be beneficial to highlight an illuminating special case. This occurs when we start with a strictly positive trigonometric polynomial $p(e^{i\theta})$ of degree $m$, and put $\phi=\log p$. By the Fej\'er-Riesz theorem, such a polynomial can be factored as $p=|q|^2$ for some analytic polynomial $q(z) =\sum_{n=0}^m a_n z^n$, in such a way that $q$ has no zeroes in the closed disk $|z|\leq 1$.  (Such a polynomial is called {\em stable}.)  In this case, we have
\[
\exp \tr \,\phi(U) =\exp \tr\, \log p(U) =|\det q(U)|^2.  
\]
For convenience, we will assume that $p$ is normalized so that $q(0)=1$. (This has the effect of setting the $c_0$ Fourier coefficient of $\log p$ equal to 0.) Then we can factor $q(z)=\prod_{j=0}^m (1+\alpha_jz)$ with each $|\alpha_j|<1$. If we let $A$ be any $m\times m$ matrix whose eigenvalues are these $\{\alpha_j\}$, and if $U$ is an $N\times N$ unitary matrix with eigenvalues $\{ \lambda_i\}$, we see that
\begin{equation}\label{eqn:intro-det-rep}
\det q(U) = \det \prod_{j=1}^m (I+\alpha_j U) = \prod_{j=1}^m\prod_{i=1}^N (1+\alpha_j \lambda_i) = \det (I+A\otimes U).
\end{equation}
where $\otimes$ is the usual Kronecker tensor product. (We refer to this identity as a {\em determinantal representation} for the stable polynomial $q$, and observe that it holds for arbitrary matrices, not just unitaries.) The Fourier coefficients of the function $\log p=\log|q|^2$ are
\[
c_n =(-1)^n\frac{\tr A^n}{n} \ (n>0), \quad c_{-n}=\overline{c_n}, \quad c_0=0;
\]
so that 
\[
\frac12 \sum_{n\in\mathbb Z} |c_n|^2 |n| = \sum_{n=1}^\infty \frac{|\tr(A^n)|^2}{n^2} n =\sum_{n=1}^\infty \frac{\tr (A\otimes\overline{A})^n}{n} = \tr \log (I-A\otimes \overline{A})^{-1}
\]
(here $\overline{A}$ denotes the entrywise complex conjugate of $A$).

So, applying the SSLT we find
\[
\lim_{N\to \infty} \mathbb E_N|\det(I+A\otimes U)|^2 = \exp\left( \tr \log (I-A\otimes \overline{A})^{-1}\right) = \det (I-A\otimes \overline{A})^{-1}.
\]
(The above manipulations are made rigorous using the fact that the spectral radius of $A$ is strictly less than $1$.) 
The argument can be extended to give a ``polarized" form:
\begin{corollary}\label{cor:intro-pencil}
    Let $A, B$ be square matrices with spectral radius $\rho(A), \rho(B)<1$. Then
\[
\lim_{N\to \infty} \mathbb E_N[\det(I+A\otimes U)\overline{\det(I+B\otimes U)]} =\det (I-A\otimes \overline{B})^{-1}.
\]    
\end{corollary}

Our goal in this paper is to prove versions of Theorem~\ref{thm:classic-sslt-prob} and Corollary~\ref{cor:intro-pencil} where we replace the single unitary matrix $U$ with a finite set of independent copies $U_1, \dots, U_g$. In the case of Theorem~\ref{thm:classic-sslt-prob} this will involve replacing $\phi$ with certain functions of noncommuting variables, and in Corollary~\ref{cor:intro-pencil} we replace $A\otimes U$ with a sum $\sum_{j=1}^g A_j\otimes U_j$.

To state our versions of Theorem~\ref{thm:classic-sslt-prob} and Corollary~\ref{cor:intro-pencil}, we need a little bit of notation. We let $\mathbb F_g$ denote the free group on $g$ generators, $\{\gamma_1, \dots, \gamma_g\}$. Elements of $\mathbb F_g$ correspond to reduced words in the letters $\{\gamma_1, \dots, \gamma_g\}$ and their inverses. The {\em length} of a reduced word is the number of letters occuring in it, denoted $|w|$. Any system of unitary operators $(U_1, \dots, U_g)$ acting in a Hilbert space defines a representation of $\mathbb F_g$ sending $\gamma_i$ to $U_i$, we let $U^w$ denote the image of the word $w$ under this representation. We let $\tr$ denote the usual (unnormalized) trace on matrices. Our first ``free" version of the SSLT is then

\begin{theorem}[``Easy" free SSLT]\label{thm:free-sslt-prob} Let $c:\mathbb F_g\to \mathbb C$ be an {\em  admissible} function satisfying $\sum_{w\in\mathbb F_g} |c_w||w|<\infty$. Let $U_1, \dots U_g$ denote independent $N\times N$ Haar distributed random unitary matrices. Then
\[
\mathbb E_N \left[ \exp \left( \tr \sum_{w\in\mathbb F_g} c_w U^w \right) \right] \sim E_0^N \cdot E_1
\]
where 
\[
E_0=\exp(c_\varnothing), \quad E_1=\exp\left( \frac12 \sum_{w\in\mathbb F_g} c_w c_{w^{-1}} |w|\right).
\]
\end{theorem}
(We have not yet defined an {\em admissible} function on $\mathbb F_g$, but this is only a technical restriction on the coefficients $c_w$, which can always be imposed without loss of generality.  See Definition~\ref{def:admissible} below, and the remark following it.) Just as above we obtain a fluctuation theorem as an immediate corollary, see Corollary~\ref{cor:free-fluctuation}. The quantity $ \sum_{w\in\mathbb F_g} c_w c_{w^{-1}} |w|$ is evidently the natural analog of the one appearing in the classical SSLT (where $g$ is $1$ and the group is $\mathbb Z$). In discussions of the classical theorem, it is often remarked that this quantity appears rather mysterious at first, so it is notable that this sum (viewed as a bilinear form on the coefficients $c_w$) emerges naturally from the free probability point of view. 

Our second theorem extends Corollary~\ref{cor:intro-pencil} to several variables. We will replace the expression $I+A\otimes U$ with a {\em monic linear pencil} in the random matrices $U_j$, that is, an expression of the form
\[
L_A(U):= I_k\otimes I_N +\sum_{j=1}^g A_j \otimes U_j.
\]
where $A_1, \dots, A_g$ are fixed $k\times k$ matrices. We let $\rho(T)$ denote the spectral radius of a matrix $T$. The following is our multivariable version of Corollary~\ref{cor:intro-pencil}, which fully proves the conjecture left open in \cite{jury2025determinantsrandomunitarypencils}.

\begin{theorem}\label{thm:intro-free-pencil} Let $A_1, \dots, A_g$, $B_1, \dots , B_g$ be systems of square matrices of size $k\times k$, $\ell \times \ell$ respectively.  Suppose that
\begin{equation}\label{eqn:intro-spectral-radius-assumption}
\rho\left(\sum_{j=1}^g A_j\otimes \overline{A_j}\right)<1, \quad \rho\left(\sum_{j=1}^g B_j\otimes \overline{B_j}\right)<1, 
\end{equation}
Then\[
\lim_{N\to \infty} \mathbb E_N \left[ \det L_A(U) \overline{L_B(U)}\right] = \det\left(I_k\otimes I_\ell - \sum_{j=1}^g A_j\otimes \overline{B_j}\right)^{-1}. 
\]
\end{theorem}
The spectral radius assumptions on $A$ and $B$ will guarantee that $\rho\left(\sum_{j=1}^g A_j\otimes \overline{B_j}\right)<1$, so the right hand side makes sense. In fact we will prove the more general Theorem~\ref{thm:free-sslt-pencil}, which also allows for quotients of determinants. In Section~\ref{sec:rational-prelim} we will extend this further to determinants of certain noncommutative rational functions of the $U_j$ (see Corollary~\ref{cor:free-sslt-rational}). Some special cases of Theorem~\ref{thm:intro-free-pencil} were proved in \cite{jury2025determinantsrandomunitarypencils}, somewhat in the spirit of the Bump-Diaconis proof of the classical SSLT mentioned above, but the combinatorial methods used there are insufficient to prove Theorem~\ref{thm:intro-free-pencil} in full generality. 

However, Theorem~\ref{thm:intro-free-pencil} and its generalizations are {\em not} straightforward applications of the general free SSLT (Theorem~\ref{thm:free-sslt-prob}). What is at stake is determining the correct domain of convergence, that is, of determining the appropriate size (spectral radius) hypothesis to place on $A$ and $B$. If one makes the stronger assumptions that
\begin{equation}\label{eqn:intro-easy-bounds}
\sup_U \|L_A(U)\|, \quad \sup_U \|L_B(U)\|\leq r<1
\end{equation}
(where the supremum is taken over all unitary matrices of all sizes) then one can expand $\log L_A(U)$ and $\log L_B(U)$ as norm convergent power series at each level $N$, and the ``easy" free SSLT can be applied. However, the spectral radius assumptions in Theorem~\ref{thm:intro-free-pencil} do not imply the bounds \eqref{eqn:intro-easy-bounds} in general (even up to a similarity), so this latter theorem is rather more subtle. In contrast, in one variable once $\rho(A)<1$ then $A$ is similar to a matrix $A^\prime$ with $\|A^\prime\|<1$ (this is a well-known theorem of Rota), so on replacing $A$ by $A^\prime$ (which does not change the determinants) we have $\|A^\prime \otimes U\|=\|A^\prime\|<1$ for all $U$, that there is nothing more to do. 

Ultimately, the mechanism that allows for this larger-than-expected domain of convergence is the so-called {\it strong convergence} (or more precisely {\it quantitative weak convergence}) phenomenon for random unitary matrices. Very loosely speaking, the idea is the following.  Consider a free $*$-polynomial $p$, which we can evaluate at systems of unitary matrices $U_1, \dots, U_g$ at any size $N$. It is well known \cite{choi, exel-loring} that supremum of $\|p(U,U^*)\|$ over all unitary matrices of all sizes is equal to the norm of $p$ in the {\em full} group $C^*$-algebra $C^*(\mathbb F_g)$ (that is, the full group $C^*$-algebra over $\mathbb F_g$ is {\em residually finite dimensional}). Thus, the $C^*(\mathbb F_g)$ norm describes the ``worst case" behavior of the norm on matrices. However, the strong convergence phenomenon for independent unitary matrices \cite[Theorem 1.4]{Collins_male_strong_asymp_free} says that the norm of $\|p(U,U^*)\|$ at size $N$ converges in probability, as $N\to\infty$, to the norm of $p(u,u^*)$ in the {\em free Haar unitaries}, that is, the norm of $p$ in the {\em reduced} group $C^*$-algebra $C^*_{r}(\mathbb F_g)$. So while the worst case behavior is governed by the norm in the full group $C^*$-algebra, in contrast, the (generically smaller)  $C^*_{r}(\mathbb F_g)$ norm captures the ``bulk" or ``typical" behavior of the norm on matrices, for large $N$. This heuristic suggest that we examine the spectral radius of $\sum_{j=1}^g A_j\otimes u_j$, which should provide an intuition for the ``typical" spectral radius of $\sum_{j=1}^g A_j\otimes U_j$ for large size $U$. We will prove the following proposition, which may be of independent interest: 
\begin{proposition}\label{prop:intro-spectral-radius}
    Let $X_1, \dots, X_g$ be $k\times k$ matrices. Then the spectral radius of the operator $\sum_{j=1}^g X_j\otimes u_j$ in $M_k(\mathbb C) \otimes C^*_{r}(\mathbb F_g)$ is equal to 
    \[
    \rho\left(\sum_{j=1}^g X_j\otimes \overline{X_j}\right)^{1/2}.
    \]
\end{proposition}

This result is not too difficult to prove given known norm estimates over the free group, but as far as we know it has not been observed in the literature up until now. 
Under our spectral radius assumption \eqref{eqn:intro-spectral-radius-assumption}, the strong convergence phenomenon will imply that even though for all $N$ there will exist $U$ with $\|\sum_{j=1}^g A_j \otimes U_j\|$ large, with high probability the spectral radius of $\sum_{j=1}^g A_j\otimes U_j$ will be less than $1$, so that $\log (I+\sum_{j=1}^g A_j\otimes U_j)$ will have a norm convergent series expansion for the ``bulk" of the $U$. However, getting explicit estimates on the contribution to $\mathbb E_N|\det L_A(U)|^2$ from the ``bad" $U$ is rather tricky, so the proof will take a more indirect route. 

\subsection{Outline of proofs and readers' guide} The key step in both proofs is getting suitably strong (uniform in $N$) tail bounds on the random variables involved. 
In the case of Theorem~\ref{thm:free-sslt-prob}, we seek to control $\mathbb E [\exp(G_N)]$ where $G_N$ is the random variable
\[
G_N:=\tr\sum_{w\in\mathbb F_g} c_w U^w.
\]
In Theorem~\ref{thm:intro-free-pencil}, if we apply the formal identity $\det A=\exp(\tr \log A)$ for positive matrices $A$, we seek to control $\mathbb E[\exp (F_N)]$, where 
\[
F_N:= \tr \left[\log L_X(U)L_X(U)^*\right].
\]
In both cases the tail bounds will ultimately follow from the concentration of measure phenomenon for the unitary group due to Meckes and Meckes \cite{ConcentrationMMeckes}, which says that for large $N$, Lipschitz functions of $U$ are nearly constant (for a precise statement, see Theorem~\ref{thm:concentration-of-measure} and Corollary~\ref{cor:com-cor} below). In the case of $G_N$ our assumption on the coefficients $c_w$ will imply that $G_N$ is Lipschitz with constant $O(\sqrt{N})$, which will allow concentration to be applied immediately. Then, the ``second order freeness" concept from free probability, due to Mingo, \'Sniady, and Speicher \cite{MingoSniadySpeicher2007}, gives the limiting values of the correlations between traces of functions of unitary matrices; this is what ultimately explains the appearance of the expression $\sum_w c_{w}c_{w^{-1}}|w|$. 

In the case of $F_N$, some care will be required since $L_X(U)L_X(U)^*$ may have a $0$ eigenvalue, or eigenvalues arbitrarily close to $0$. In particular, concentration of measure cannot be applied immediately, since these functions $F_N$ will in general not obey suitable Lipschitz estimates. So instead of getting tail bounds for $F_N$ directly, we introduce a smooth cutoff of the $\log$ function in the definition of $F_N$, and control the expectations of these cut-off random variables. This will necessitate an appeal to strong convergence (specifically, an estimate of Parraud \cite{Parraud2022}, which in turn will also require control of the spectral radius of matricial sums of free Haar unitaries $\sum_{j=1}^g X_j \otimes u_j$). Once we have suitable tail bounds, we can appeal to the ``easy" free SSLT to show that the claimed limit is valid for small $X$, then use the tail bounds to enlarge the domain of $X$ for which the expectations remain bounded. The result then follows by an analytic continuation and compactness argument (in the guise of Montel's theorem). 

Section~\ref{sec:prelim} contains a review of the preliminary facts we need; these include: basic facts about noncommutative functions and matrix analysis, the unitary group $\mathcal U(N)$ and the free group von Neumann algebra $L(\mathbb F_g)$, the concentration of measure phenomenon for $\mathcal U(N)$, free probability and second order freeness, strong convergence and Parraud's estimate, and some elementary lemmas. 

Section~\ref{sec:easy-free-sslt} contains the proof of the ``easy" free SSLT (Theorem~\ref{thm:free-sslt-prob}, proved there as Theorem~\ref{thm:a_free_szego}), and the corresponding fluctuation theorem (Corollary~\ref{cor:free-fluctuation}). Section~\ref{sec:uniform-tail-bounds} proves a key technical theorem giving uniform tail bounds which will be needed eventually for the proof of Theorem~\ref{thm:intro-free-pencil}. Section~\ref{sec:spectral-radius} contains the proof of the spectral radius formula Proposition~\ref{prop:intro-spectral-radius}, proved there as part of the more general Theorem~\ref{thm:spectral-radius}. This section is self-contained and can be read independently of the others.  Section~\ref{sec:free-sslt-pencil} combines the results of the previous three sections to prove the crucial uniform bounds on expected determinants (Proposition~\ref{prop:uniform-bound-l2}), and proves Theorem~\ref{thm:free-sslt-pencil}, which is a more general version of Theorem~\ref{thm:intro-free-pencil} stated above, allowing for both products and quotients of determinants. In Section~\ref{sec:rational-prelim}, we begin with the necessary material on free rational functions and determinantal representations, which will then be used to prove Corollary~\ref{cor:free-sslt-rational}, which extends Theorem~\ref{thm:free-sslt-pencil} to determinants of free rational functions of $U$.

\section{Preliminaries and notation}\label{sec:prelim}

\subsection{Noncommutative Polynomials}
Let $\mathbb{F}^+_g$ denote the free monoid on $g$ symbols $\{1, \dots, g\}$. We define the algebra of noncommutative polynomials in $g$ indeterminates
$x = (x_1,\dots,x_g)$ by
\[
\C\langle x\rangle := \C \mathbb{F}_g^+,
\]
where recall $\C S$ denotes all formal $\C$-linear combinations of some set $S$. 

To incorporate a $*$-operation, we also consider the free monoid
$\mathbb{F}_{2g}^+$ on $2g$ generators. For $1 \le i \le g$, define
\[
x_i^* := x_{i+g},
\qquad
x_{i+g}^* := x_i \quad (1 \le i \le g).
\]
For a word $w = i_1 \cdots i_n \in \mathbb{F}_{2g}^+$, define
\[
(x^w)^* := (x_{i_1} \cdots x_{i_n})^*
:= x_{i_n}^* \cdots x_{i_1}^*.
\]
If
\[
p = \sum_{w \in \mathbb{F}_{2g}^+} a_w x^w \in \C \mathbb{F}_{2g}^+,
\]
we define
\[
p^* := \sum_{w \in \mathbb{F}_{2g}^+} \overline{a_w} (x^w)^*.
\]

Rather than listing all $2g$ generators explicitly, we denote
$x_{g+1},\dots,x_{2g}$ by $x_1^*,\dots,x_g^*$ and write
\[
\C\langle x, x^* \rangle := \C \mathbb{F}_{2g}^+.
\]
Every polynomial $p \in \mathbb{C}\langle x, x^\ast\rangle$ admits a unique expansion
\[
p = \sum_{w \in \mathbb{F}^+_{2g}} a_w\, x^w,
\]
where the coefficients $p_w \in \mathbb{C}$ and only finitely many of them are nonzero. And we say $p$ is 
\emph{formally self-adjoint} to mean that $p$ is self-adjoint with respect to the $*$-algebra $\mathbb{C}\langle x, x^\ast\rangle$. In particular this will mean that $p(a_1,\cdots,a_g)$ will be self-adjoint when evaluated on any tuple of elements $a_1,\cdots,a_g$ in a $*$-algebra $\mathcal{A}$. We also define the \emph{degree} of a nc polynomial $p\in M_{k}(\C)\lip x,x^*\rip$ by 
\[\deg(p):= \max\{|w|: a_w\not=0\}\]
where $|w|$ denotes the length of the word $w$.

We do not wish to restrict our nc polynomials to have only scalar coefficients; instead, we allow matrix coefficients of arbitrary size.
For \(n\in\mathbb{N}\), we denote by \(M_n(\mathbb{C}\langle x\rangle)\) the algebra
of \(n\times n\) matrices whose entries are nc polynomials in
\(\mathbb{C}\langle x\rangle\).

There is a canonical identification of vector spaces
\[
M_n(\mathbb{C}\langle x\rangle)
\;\cong\;
M_n(\mathbb{C}) \otimes \mathbb{C}\langle x\rangle.
\]

By $M_n(\C)\lip x,x^*\rip$ we refer to any one of the above identifications. One also obtains a $*$-operation 
\[
p=\sum_{w} A_w \otimes x^w \in M_n(\mathbb{C}) \otimes \mathbb{C}\langle x\rangle
\]
by defining
\[
p^* := \sum_{w} A_w^* \otimes x^{w*}.
\]
We define the notions of \emph{formal self-adjointness} and of \emph{degree} exactly as in the scalar-coefficient case.

If \(p\in M_k(\mathbb{C}\langle x,x^*\rangle)\) and
\(X=(X_1,\dots,X_g)\in M_N(\mathbb{C})^g\), then the evaluation
\[
p(X,X^*)
\]
is an element of
\[
M_k(\mathbb{C}) \otimes M_N(\mathbb{C}),
\]
where the tensor product is identified with the Kronecker  product
of matrices.

More generally, if \(a_1,\dots,a_g\) belong to a \(*\)-algebra \(\mathcal{A}\),
then evaluation yields
\[
p(a_1,\dots,a_g) \in M_k(\mathbb{C}) \otimes \mathcal{A}.
\]

\begin{definition}\label{def:bounded-family}
    We say that a collection of polynomials $\mathcal P\subset  M_k(\mathbb{C}\langle x,x^*\rangle)$ is a {\em bounded family} if the degrees of the elements $p\in\mathcal P$ are uniformly bounded, and there is a uniform bound on the norms of the coefficients of $p\in \mathcal P$.
\end{definition}

\subsection{The unitary group and concentration of measure}
The set $\mathcal U(N)\subset M_N(\mathbb C)$ of $N\times N$ unitary matrices is a compact Lie group, and comes equipped with a (unique) invariant probability measure, the Haar measure. We will generally use probabilistic notation and terminology around integration on $\mathcal U(N)$, so we will write $\mathbb E[f(U)]$ for the integral $\int f(U)\, dU$, and $\mathbb P(E)$ for the probability of an event (measure of set), etc. In general, to avoid over-burdening the notation, we will usually drop the explicit $N$-dependence and  write $U$ in place of $U^{(N)}$; the capital letter $U$ will always be understood to refer to a unitary {\em matrix}. This should not cause any confusion. Similarly we will usually drop the $N$ from the expectation $\mathbb E_N$ and the trace $\tr_N$, unless we wish to emphasize it for clarity.

We refer to the book \cite{Meckes_2019_RMT_compact_group} for background material on $\mathcal U(N)$. The group $\mathcal U(N)$ inherits the Hilbert-Schmidt metric from its inclusion in $M_N(\mathbb C)$, and we equip the $g$-fold product $\mathcal U(N)^g$ with the metric 
\[
d(U,V):= \left( \sum_{i=1}^g \|U_i - V_i\|_{2}^2 \right)^{1/2},
\]
where $\|A\|_2^2:= \tr(A^*A)$. When we speak of Lipschitz functions on $\mathcal U(N)^g$, we always mean Lipschitz with respect to this metric. 

We will make repeated use of the concentration-of-measure phenomenon for $\mathcal U(N)$, expressed in the following theorem \cite[Corollary 17]{ConcentrationMMeckes}: 

\begin{theorem}[Concentration of measure for $\mathcal U(N)$]    \label{thm:concentration-of-measure}

Let $F : \mathcal{U}(N)^g \to \mathbb{R}$ be a Lipschitz function with Lipschitz constant $L$, and let $U_1,\ldots,U_g$ be independent Haar-distributed $N\times N$ unitary random matrices.
Then for every $t > |\mathbb{E}F(U_1,\ldots,U_g)|$,
\[
\mathbb{P}\;\!\bigl(\; |F(U_1,\ldots,U_g)| \ge t \;\Bigr)
\le 
2 \exp\!\left( \frac{ -N\bigl(t - |\mathbb{E}F(U_1,\ldots,U_g)|\bigr)^2}{12L^2} \right).
\]
\end{theorem}

We record a simple corollary which we will use frequently:

\begin{corollary}\label{cor:com-cor}
    Let $F_N:\mathcal U(N)^{(g)}\to \mathbb C$ be a sequence of functions, each of which is Lipschitz with Lipschitz constant at most $L\sqrt{N}$. Assume further that the $F_N$ are centered ($\mathbb EF_N=0$ for each $N$). Then the $F_N$ obey a uniform sub-Gaussian tail bound: for all $t>0$
    \begin{equation}
        \mathbb P(|F_N|>t) \leq 4\exp\left( \frac{-t^2}{24L^2}\right).
    \end{equation}
\end{corollary}
\begin{proof}
     The real and imaginary parts $\Re F_N$ and $\Im F_N$ are centered and obey the same Lipschitz estimates as $F_N$. From the usual union bound
 \begin{equation*}\label{eq:re-im-triangle}
    \mathbb{P}\left(\left|F_N\right|>t\right) \leq \mathbb{P}\left(\left|\Re F_N\right|>t/\sqrt{2}\right) +  \mathbb{P}\left(\left|\Im F_N\right|>t/\sqrt{2}\right)
\end{equation*}
we conclude from Theorem~\ref{thm:concentration-of-measure} that 
\begin{align*}
    \mathbb{P}\left(\left|F_N\right|>t\right) & \leq 2\exp\left(\frac{-N(t/\sqrt{2})^2}{12(L\sqrt{N})^2}\right)+ 2\exp\left(\frac{-N(t/\sqrt{2})^2}{12(L\sqrt{N})^2}\right)\\
    &=  4\exp\left(\frac{-t^2}{24L^2}\right).
\end{align*}
\end{proof}

\subsection{The free group and free probability}

Fix an integer $g \geq 1$, and let $\mathbb{F}_g$ denote the \emph{free group} on $g$ generators
\[
\{\gamma_1,\gamma_2,\dots,\gamma_g\}.
\]
which consists of all words in these letters and their inverses. The identity element corresponds to the empty word.
For a reduced word $w \in \mathbb{F}_g$, the number of letters appearing in $w$ is called the
\emph{length} of $w$ and is denoted by $|w|$. Denote $W_n = \{w\in \mathbb{F}_g: |w|=n\}$.

Let
\[
\lambda : \mathbb{F}_g \longrightarrow \mathcal{U}\bigl(\ell^2(\mathbb{F}_g)\bigr)
\]
denote the \emph{left regular representation} of $\mathbb{F}_g$.
The \emph{reduced $C^*$-algebra} of the free group is
\[
C^*_\lambda(\mathbb{F}_g)
:= C^*\bigl(\lambda(\mathbb{F}_g)\bigr)
\subseteq \mathcal{B}\bigl(\ell^2(\mathbb{F}_g)\bigr)
\]
and let $\{\delta_h\}_{h\in \mathbb{F}_g}$ denote the standard basis for $\ell^2(\mathbb{F}_g)$.

For $i=1,\dots,g$, set
\[
u_i := \lambda(\gamma_i).
\]
The operators $u_1,\ldots,u_g$ are freely independent Haar unitaries, in the sense of Voiculescu’s free probability theory, with respect to the canonical tracial state $\tau$ on $C^*_\lambda(\mathbb{F}_g)$.

\subsubsection{Asymptotic freeness and second order freeness}\label{sec:second-order}

The language and tools of {\em free probability} will be essential to our results. There are many sources, but all the material we need may be found in the book of Mingo and Speicher \cite{mingo2017free}. We will describe only the minimal material we will need, and refer to the text for more background. 

A {\em free probability space} is pair $(\mathcal A, \varphi)$ where $\mathcal A$ is a complex, unital $*$-algebra, and $\varphi:\mathcal A\to \mathbb C$ is a linear functional such that $\varphi(1)=1$. In our case $\varphi$ will always be {\em tracial} ($\varphi(ab)=\varphi(ba)$) and {\em positive} ($\varphi(a^*a)\geq 0$ for all $a\in \mathcal A$). A {\em second order free probability space} is a triple $(\mathcal A, \varphi, \varphi_2)$ where $(\mathcal A, \varphi)$ is a free probability space, and $\varphi_2:\mathcal A\times \mathcal A\to \mathbb C$ is a symmetric, bilinear functional which is tracial in both arguments and which satisfies 
$\varphi_2(a,1)=0=\varphi_2(1,b)$ for all $a,b\in \mathcal A$. 

In practice, $\varphi$ often occurs as the limit of normalized traces of random matrices (in a suitable sense), and $\varphi_2$ arises as the limit of covariances of pairs of {\em un-}normalized traces. We will describe the details in the unitary setting, which is the only one we will use in this paper. 

We let $u_1, \dots, u_g\in C^*_{\lambda}(\mathbb F_g)$ be the standard freely independent Haar unitaries, as above; and we let $\mathcal A$ be the *-algebra generated by the $u_j$. Taking the functional $\varphi$ to be the canonical trace $\tau:x\to \langle x\delta_e, \delta_e\rangle$ as above, $(\mathcal A, \varphi)$ is then a free probability space. To introduce $\varphi_2$, we need a couple of definitions. 

On words $w\in\mathbb F_g$ of length $n$ there is a {\em rotation map}
\[
\gamma_{i_1} \gamma_{i_2} \cdots \gamma_{i_n} \to \gamma_{i_n} \gamma_{i_1} \gamma_{i_2} \cdots \gamma_{i_{n-1}}.
\]
A {\em cyclic rotation} of $w$ is a word obtained by finitely many iterates of the rotation map, this induces an action of the cyclic group of order $n$ on the set of words of length $n$. A word $w$ is called {\em cyclically reduced} if every cyclic rotation of $w$ is reduced. We observe that $w$ is cyclically reduced if and only if $w^{-1}$ is cyclically reduced. We let $orb(w)$ denote the orbit of $w$ under rotations, that is, the set of all distinct words obtained as rotations of $w$. (For example, the word $w=aaab$ admits four distinct rotations, while $w=abab$ admits only two.) We let $j(w)$ denote the number of elements of $\mathbb Z_n$ leaving $w$ unchanged, i.e. $j(w)=|stab(w)|$. By the orbit-stabilizer theorem $|orb(w)|=\frac{|w|}{j(w)}$. 
(The reason to bother about all of this is that the trace monomial $\tr(U^w)$ is invariant under rotations of $w$.)

For a pair of cyclically reduced words $w,v$, we now define $\varphi_2$ as
\begin{equation}\label{eqn:phi2}
    \varphi_2(u^w, u^v) =\begin{cases} j(w) & \text{if } v \text{ is a cyclic rotation of } w^{-1} \\ 0 & \text{otherwise} \end{cases}
\end{equation}
Using bilinearity and the tracial property, it is clear that $\varphi_2$ extends uniquely to $\mathcal A\times \mathcal A$ in such a way that $(\mathcal A, \varphi, \varphi_2)$ is a second order free probability space.

A celebrated theorem of Voiculescu \cite{Voiculescu1991} says that for independent Haar unitary matrices $U_1^{(N)}, \dots, U_g^{(N)}$, we have for every word $w$
\begin{equation}
\lim_{N\to \infty} \frac1N \tr (U^{(N)w}) = \varphi(u^w).
\end{equation}
almost surely. One of the key results of  Mingo, \'{S}niady, and Speicher \cite{MingoSniadySpeicher2007} is a second order version, which says that the covariances of {\em un-}normalized traces converge to $\varphi_2$; precisely, for all cyclically reduced words $w,v$ we have
    \begin{equation}\label{eqn:second-order-freeness}
        \lim_{N\to \infty}\mathbb E[\tr(U^w)\tr(U^v)] = \varphi_2(u^w, u^v)
    \end{equation}
Moreover, they prove that for $p\in \mathbb C\langle x,x^*\rangle$ with $0$ constant term, the random variables $\tr\, p(U, U^{*})$ converge in distribution to a centered complex Gaussian random variable $Z$, with variance $\varphi_2(p(u,u^*), p(u,u^*))$. (The right way to understand all this is through the mechanism of {\em free cumulants}, which we do not introduce here, since we do not need them explicitly, but refer to \cite{mingo2017free}.) 

We find it helpful to think of (\ref{eqn:second-order-freeness}) as an asymptotic orthogonality relation: if we make the change of variables $v\to v^{-1}$, and observe that 
\[
\tr\, U^{v^{-1}} =\tr\, U^{v*} =\overline{\tr\, U^v}
\]
then on combining (\ref{eqn:second-order-freeness}) and (\ref{eqn:phi2}) we have

 \begin{equation}\label{eqn:orthogonality}
        \lim_{N\to \infty}\mathbb E[\tr(U^{w})\overline{\tr(U^{v})}] = \delta_{wv} j(w).
    \end{equation}
\begin{remark}\label{rem:no-bump}
The result can be upgraded to give asymptotic orthogonality relations for products $\tr(U^{w_1})\cdots \tr(U^{w_n})$; these same relations were discovered independently by R\u{a}dulescu \cite{radulescu2006combinatorial}, but we will not need to make explicit use of these more general formulas. (These more general orthogonality relations were used in \cite{jury2025determinantsrandomunitarypencils} to prove some special cases of Theorem~\ref{thm:intro-free-pencil}, but the fuller second order freeness machinery seems better suited to the present context.)
We also note that in one variable, these asymptotic orthogonality relations hold in a much sharper way, in that they become exact once $N$ is sufficiently large (see \cite[Theorem 2.1]{diaconis-evans}); this fact underlies the proof of SSLT given by Bump and Diaconis \cite{bump-diaconis-2002}. The lack of exact orthogonality in the multivariate setting is what prevents that proof from going through in our context.
\end{remark}

\subsubsection{The Fuglede-Kadison determinant} The trace $\tau$ on the free group factor $L(\mathbb{F}_g)$ and normalized matrix trace $\frac1k \tr_k$ gives us a normalized trace $\frac{1}{k}\tr_k\otimes \tau$ on $M_k(\C)\otimes L(\mathbb{F}_g)$. With respect to this trace we can define the \emph{Fuglede-Kadison} determinant of  invertable elements   in $M_k(\C)\otimes L(\mathbb{F}_g)$ by 
\[
\Delta_{FK}(a):=\exp\left[(\frac{1}{k}\tr_k\otimes \tau)(\log |a|)\right]
\]
    where $\log(|a|)$ well-defined by the functional calculus, since by the assumption that $a$ is invertible, the spectrum of $|a|=(a^*a)^{1/2}$ is in $(0,\infty)$. One can extend the definition of the Fuglede-Kadison determinant to non-invertable elements via a regularization procedure, but this is not needed for our purposes. We record the following basic properties of $\Delta_{FK}$, we refer the reader to \cite[Section 2]{Haagerup_Larsen} or \cite[Section 11.3]{mingo2017free}. In our application the tracial von Neumann algebra will be the free group factor $L(\mathbb F_g)$ (the von Neumann algebra generated by $C^*_\lambda(\mathbb F_g)\subset \mathcal B (\ell^2(\mathbb F_g))$). 
\begin{lemma}[Properties of $\Delta_{FK}$]\label{lem:FK-determinant}
Let $a, b$ be elements of a tracial von Neumann algebra $(M, \tau)$. Then
\begin{enumerate}
    \item $\Delta_{FK}(ab)=\Delta_{FK}(a)\Delta_{FK}(b)$
    \item (\cite[Lemma 4.3]{Haagerup_Larsen}) If $a$ has spectral radius $\rho(a)<1$ and $\tau(a^n)=0$ for all $n>0$,  then $\Delta_{FK}(1+a) =1$.
\end{enumerate}
\end{lemma}

\subsubsection{Quantitative weak convergence}
It will be essential for us to control traces of $f(p(U, U^*))$ for smooth functions $f$, more general than polynomials. 
Given $f:\R\ra\R$ at least $k$-times boundedly differentiable,  we define
\[\|f\|_{C^k}:= \sum_{n=0}^k \|f^{(n)}\|_{\infty}.\]

We have the following theorem of Parraud \cite[Theorem 1.1]{Parraud2022} (which we state in a somewhat simplified form, since we do not need its full generality). Note that our parameter $k$ is $M$ in Parraud's statement, and the $Z^{NM}$ matrices for us have the form $Z=I_N\otimes Y_k$, so that altogether $P(U\otimes I_N, Z^{Nk})$ defines a polynomial with $k\times k$ coefficients, evaluated at the $U$. Lastly, in our notation \(C(p)\) is exactly the polynomial with positive real coefficients \(L_p\) introduced by Parraud, evaluated at the operator norm of the largest matrix coefficient of \(p\).
(For the ``moreover" portion of the theorem statement, see the remark following \cite[Theorem 4.1]{Parraud2022}.)
\begin{theorem}[\cite{Parraud2022}]\label{thm:parraud-asymptotics}

Let $U_1^{(N)},\cdots, U_g^{(N)}$ be independent, $N\times N$, Haar distributed,  unitary random matrices, and $u_1,\cdots,u_g$ free Haar unitary operators. 

Suppose $p\in M_k(\C)\lip x,x^*\rip$ is a formally self-adjoint nc polynomial, and $f:\R\ra\R$ at least 6 times boundedly differentiable. Then there exists a constant $C(p)>0$ (only depending on $p$) such that for all $N>1$,  
\begin{equation}\label{eq:parraud-eq}
    \left| \mathbb{E}_N\!\left[\frac{1}{kN}\tr_{kN}\;f(p(U^{(N)},U^{(N)*})) \right]- \left(\frac{1}{k}\tr_k\otimes \tau\right)f(p(u,u^*)) \right|\leq \frac{\log(N)^2k^2}{N^2}C(p)\|f\|_{C^6}.
\end{equation}

Moreover if $\mathcal{P}$ is a self-adjoint and bounded family then there exists a single constant $C(\mathcal{P})>0$ such that equation~\eqref{eq:parraud-eq} holds with $C(\mathcal{P})$ for all $p\in \mathcal{P}$.
\end{theorem}

\subsection{Lipschitz estimates}

In this subsection we record some elementary Lipschitz estimates on matrix functions which will be used repeatedly in what follows. 

We let $\|\cdot\|$ denote the usual operator norm of a matrix, and $\|A\|_2:=[\tr A^*A]^{1/2}$ the Hilbert-Schmidt norm.  For any matrices $A$ and $B$ we have $\|AB\|_2\leq \|A\|\|B\|_2$.

For matrices $A$ and $B$ we can write $A^n-B^n=\sum_{j=0}^{n-1} A^j (A-B)B^{n-j-1}$. From this we obtain the standard estimate
\begin{lemma}\label{lem:An-Lipschitz} If $\|A\|, \|B\|\leq r$, then $\|A^n-B^n\|_2\leq nr^{n-1} \|A-B\|_2$. 
\end{lemma}
The important feature of the next two lemmas is that the Lipschitz constants do not depend on the size $N$ of the unitary matrices. 
  
By a similar trick to that used in the previous lemma one sees easily that
\begin{lemma}[Monomial Lipschitz Constants]\label{eq:monomial-lipschitz-bound}
    For any word $w\in \mathbb{F}_g$, $U,V\in \mathcal{U}(d)^g$,
    \begin{equation*}
        \|U^w - V^w\|_{2}\leq |w|d(U,V)
    \end{equation*}
    In particular if $\sum_{w\in\mathbb F_g} |c_w||w|=L<\infty$ then the function $(U_1, \dots, U_g)\to \sum_{w\in\mathbb F_g} c_w U^w$ is Lipschitz (with respect to the Hilbert-Schmidt metric on both sides) with constant at most $L$. 
\end{lemma}

\begin{lemma}
If $X_1, \dots X_g$ are $k\times k$ matrices, then the map
\[
(U_1, \dots, U_g)\to \sum_{j=1}^g X_j\otimes U_j
\]
is Lipschitz from $\mathcal U(N)^g$ to $(M_{kN}(\mathbb C), \|\cdot\|_2)$, with the estimate
\[
\left\|\sum_{j=1}^g X_j\otimes (U_j-V_j)\right\|_2 \leq \sqrt{k} \|X\|_{row} d(U,V).    
\]
\end{lemma}
Here $\|X\|_{row}:=\|X_1X_1^*+\cdots X_gX_g^*\|^{1/2}$. 
\begin{proof}Define a row matrix $A$ and a column matrix $B$ by 
\[
A= \begin{pmatrix} X_1\otimes I_N & \cdots & X_g\otimes I_N\end{pmatrix}, \quad B= \begin{pmatrix} I_k\otimes (U_1-V_1) \\ \vdots \\ I_k \otimes (U_g-V_g)\end{pmatrix}
\]
and apply the estimate $\|AB\|_2\leq \|A\|\|B\|_2$. 
\end{proof}

Next, we observe that the trace satisfies $|\tr(A)|\leq \sqrt{N} \|A\|_2$ for $N\times N$ matrices $A$ (apply Cauchy-Schwarz with respect to the tracial inner product).  

\begin{lemma}\label{lem:log-lipschitz} For $N\times N$ matrices $A, B$ with $\|A\|, \|B\|\leq r<1$, we have
\[
|\tr \log (I+A) -\tr\log(I+B)| \leq \sqrt{N}\frac{1}{1-r} \|A-B\|_2.
\]
\end{lemma}
\begin{proof} We have the norm convergent power series expansion
\[
\log(I+A) =\sum_{n=1}^\infty \frac1n A^n.
\]
The estimate follows immediately from Lemma~\ref{lem:An-Lipschitz} and the above remark about the trace. 
\end{proof}

Finally we need the following fact from matrix analysis, see for example \cite[Lemma VII.5.5]{bhatia2013matrix}.
\begin{lemma}\label{lem:lipschitz-on-sa}
If $f:\mathbb R\to \mathbb R$ is Lipschitz with constant $L$, and $A, B$ are self-adjoint matrices, then 
\[
\|f(A)-f(B)\|_2 \leq L\|A-B\|_2.
\]
\end{lemma}

\section{The ``easy" free SSLT and a fluctuation theorem}\label{sec:easy-free-sslt}

\begin{definition}\label{def:admissible}
We will call a function $c:\mathbb F_g\to \mathbb C$ {\em admissible} if:
\begin{itemize}
\item[i)] $c$ is supported on the set of cyclically reduced words, and
\item[ii)] $c_w=c_v$ whenever $v$ is a cyclic rotation of $w$ (i.e the function $c$ is constant on orbits).
\end{itemize}
\end{definition}
{\bf Remark:} While the admissibility assumption is essential to the proofs of the results in this section, it is not really a limitation, since it is easy to check that if $c$ is any function with $\sum_{w\in\mathbb F_g} |c_w||w|<\infty$, then by averaging we can create an admissible function $c^\prime$ which still satisfies the summability condition, and with the property that
\[
\sum\limits_{w\in \mathbb{F}_g}c_w\tr(U^w) = \sum\limits_{w\in \mathbb{F}_g}c^\prime_w \tr(U^w) 
\]
for all unitary matrices $U$.

We recall the notation from Section~\ref{sec:second-order}.
\begin{lemma}\label{lem:second-order-moments-comp}
    If $c:\mathbb{F}_g\ra\C$ is admissible, finitely supported, and $c_{\varnothing}=0$, then
    \begin{equation*}\label{eq:second-moments}
        \varphi_2\left(\sum_{w\in \mathbb{F}_g}c_wu^w,\;\sum_{w\in \mathbb{F}_g}c_wu^w\right) = \sum_{w\in \mathbb{F}_g}c_{w}c_{w^{-1}}|w|.
    \end{equation*}
\end{lemma}
\begin{proof}
    The fact that $c$ is constant on orbits under cyclic rotation means that $c_v=c_{w^{-1}}$ for every $v\sim w^{-1}$. By the orbit-stabilizer theorem (applied to the action of the cyclic group $\mathbb Z_n$ on the set of words of length $n$), the number of such $v$ is equal to $\frac{|w|}{j(w)}$. Thus for admissible, finitely supported $c$
    \[
   \varphi_2\left(\sum_{w\in\mathbb F_g} c_w u^w, \sum_{w\in\mathbb F_g} c_w u^w\right) = \sum_{w,v\in\mathbb F_g, \ v\sim w^{-1}} c_wc_v j(w) = \sum_{w\in \mathbb{F}_g}c_{w}c_{w^{-1}}|w|.
    \]
\end{proof}
\begin{theorem}[The ``easy" free SSLT]\label{thm:a_free_szego}
  Let $c:\mathbb F_g\to \mathbb C$ be an admissible function which satisfies \[\sum_{w\in \mathbb{F}_g}|c_w||w|<\infty,\] and let $G_N$ denote the random variable 
  \[ G_N:= \tr \sum_{w\in \mathbb F_g} c_w U^{(N)w}.\]
    Then 
    \begin{equation*}
        \mathbb E \left[ \exp(G_N)\right]\sim E_0^N\cdot E_1
    \end{equation*}  
where
\[
E_0=\exp(c_\varnothing), \quad E_1 = \exp\left(\frac12 \sum\limits_{w\in \mathbb{F}_g}c_wc_{w^{-1}}|w|\right).
\]
\end{theorem}

\begin{proof} Since $\mathbb EG_N = Nc_\varnothing$, it will suffice to assume that $c_\varnothing=0$ and prove that
  \begin{equation*}
        \lim\limits_{N\ra\infty}
        \mathbb E \left[ \exp(G_N)\right]= \exp\left(\frac12 \sum\limits_{w\in \mathbb{F}_g}c_wc_{w^{-1}}|w|\right).
    \end{equation*}
We first assume that the function $c_w$ is finitely supported. In this case, since independent random unitaries are free of second order \cite[Theorem 3.13]{MingoSniadySpeicher2007}, it follows that 
    \begin{equation*}
         \sum_{w\in \mathbb F_g} c_w\; \tr\, U^{w} \xrightarrow[N\ra\infty]{\text{distribution}} Z
    \end{equation*}
where $Z$ is a centered complex Gaussian random variable, and its cumulants are obtained as limits of the corresponding cumulants of $G_N$. Here we rely on the fact that $G_N$ is a linear combination of random trace monomials $\tr \, U^{ w}$ ranging over cyclically reduced words $w$, which follows from the admissibility assumption.

In particular 
\begin{align}\label{eqn:exp-z-square}
\begin{split}
    \mathbb{E}\left[Z^2\right] &= \lim_{N\ra\infty}\mathbb E\left[\ \left(\sum_{w\in \mathbb F_g} c_w \;\tr U^{w}\right)\left(\sum_{w\in \mathbb F_g} c_w \;\tr U^{w}\right)\right]\\
                                &=\varphi_2\left(\sum_{w\in \mathbb{F}_g}c_wu^w,\;\sum_{w\in \mathbb{F}_g}c_wu^w\right)
\end{split}
\end{align}
where the last equality follows from second order freeness (\ref{eqn:second-order-freeness}) and bilinearity. 

By an elementary probability calculation, for a centered complex Gaussian $Z$ one has  $\mathbb{E}\left[\exp(Z)\right] = \exp\left(\frac{1}{2}\mathbb{E}\left[Z^2\right]\right)$. Then
\begin{align*}
    \mathbb{E}\left[\exp(Z)\right] &= \exp\left(\frac{1}{2}\mathbb{E}\left[Z^2\right]\right)\\
        &= \exp\left(\frac{1}{2}\varphi_2\left(\sum_{w\in \mathbb{F}_g}c_wu^w,\;\sum_{w\in \mathbb{F}_g}c_wu^w\right)\right)\\
        &= \exp\left(\frac12 \sum_{w\in\mathbb F_g} c_wc_{w^{-1}}|w|\right)
\end{align*}
by (\ref{eqn:exp-z-square}) and Lemma~\ref{lem:second-order-moments-comp}.

To finish the proof in the finitely supported case, we show that the sequence $\exp(G_N)$ is uniformly integrable. The random variables $G_N$ are centered, and each $G_N$ is Lipschitz on $\mathcal U(N)^g$, with Lipschitz constant bounded by $L\sqrt{N}$ (where $L=\sum_{w\in\mathbb F_g} |c_w||w|<\infty$). It then follows from concentration of measure (Corollary~\ref{cor:com-cor}) that we have the sub-Gaussian tail bound
\begin{equation}
      \mathbb{P}\left(\left|G_N\right|>t\right) \leq  4\exp\left(\frac{-t^2}{24L^2}\right).
\end{equation}
This implies that the sequence $\exp(G_N)$ is uniformly integrable, and therefore it will converge in distribution to $\exp(Z)$. Moreover, uniform integrability will also guarantee convergence of the expectations of $\exp(G_N)$:
\begin{equation*}
    \lim_{N\ra\infty}\mathbb{E}\left[\exp\left(G_N\right)\right] = \mathbb{E}\left[\exp(Z)\right].
\end{equation*}
This completes the proof in the case that $c_w$ is finitely supported.     

For the general case let $N,M\in \mathbb{N}$, and split the random variable $G_N= H_{N,M}+T_{N,M}$ where
\[
H_{N,M} = \tr \sum_{|w|\leq M} c_w U^{(N)w}, \quad T_{N,M} = \tr \sum_{|w|> M} c_w U^{(N)w}.
\]
Now consider
\begin{align*}
\begin{aligned}
&\left|
   \mathbb{E}\!\big[\exp G_N\big]
   - \exp\!\left(\frac12 
     \sum_{w\in\mathbb{F}_g}c_w c_{w^{-1}}\,|w|\right)
 \right|
\\[0.4em]
&\le \
\left|
  \mathbb{E}\!\left[\exp G_N - \exp H_{N,M}\right]
\right|
\\[0.4em]
&\quad
+ \left|
    \mathbb{E}\!\Bigl[\exp H_{N,M}\Bigr]
    - \exp\!\left(\frac12
      \sum_{|w|\le M}c_w c_{w^{-1}}\,|w|\right)
  \right|
\\[0.4em]
&\quad
+ \left|
    \exp\!\left(\frac12
      \sum_{|w|\le M}c_w c_{w^{-1}}\,|w|\right)
    - \exp\!\left(\frac12
      \sum_{w\in\mathbb{F}_g}c_w c_{w^{-1}}\,|w|\right)
  \right|.
\end{aligned}
\end{align*}

Examining the three terms on the right hand side of the inequality, we can choose $M$ large enough to make the last (deterministic) term arbitrarily small (observe that the series $\sum_{w\in\mathbb F_g} c_wc_{w^{-1}}|w|$ is absolutely convergent under the assumption $\sum_{w\in\mathbb F_g} |c_w||w|<\infty$). Now, for that choice of $M$, the middle term can be made arbitrarily small for $N$ large enough, from the first part of the proof.  Therefore it remains to show
\[
  \mathbb E  \left| \exp G_N- \exp H_{N, M}  \right| 
\]

goes to $0$ \emph{uniformly in $N$} as $M\ra \infty$. This will be accomplished by another application of concentration of measure. 

By factoring and Cauchy-Schwarz we have
\begin{align*}
    \mathbb E  \left| \exp G_N- \exp H_{N, M}  \right| 
    &= \mathbb E \left| \exp G_N \cdot \left(1-\exp (-T_{N,M})\right)\right| \\
    &\leq \left(\mathbb E |\exp 2G_N|\right)^{1/2} \left(\mathbb E |1-\exp (-T_{N,M})|^2 \right)^{1/2} 
\end{align*}
Now the random variables $G_N$ are centered and have $O(\sqrt{N})$ Lipschitz constants, so by concentration of measure (Corollary~\ref{cor:com-cor}), $G_N$ obeys a uniform sub-Gaussian tail bound, which implies that $\mathbb E |\exp 2G_N|$ is uniformly bounded in $N$. So, it remains to show $\mathbb E |1-\exp (-T_{N,M})|^2$ goes to $0$ as $M\to \infty$, uniformly in $N$. We begin by observing the trivial estimate
\[
 \mathbb E |1-\exp (-T_{N,M})|^2 \leq \mathbb E\left( \exp( |T_{N,M}|)-1\right)^2.
\]
Now let $\varphi(t):=(e^t-1)^2$. Then from elementary probability,

\begin{equation}\label{eq:layer-cake}
\mathbb{E}\Bigl[\Bigl(\exp(|T_{N,M}|) - 1\Bigr)^2 \Bigr]
= \int_{0}^{\infty} \varphi'(t)\,
\mathbb{P}\!\left(|T_{N,M}|>t\right)\,dt.
\end{equation}

Let $\delta_M = \sum_{|w|>M} |c_w||w|$. By its definition we seen that $T_{N,M}$ has Lipschitz constant bounded by $\delta_M \sqrt{N}$. By Corollary~\ref{cor:com-cor} again, we conclude that the $T_{N,M}$ obey a sub-Gaussian tail bound of the form
\begin{equation}\label{ineq:com_final_claim}
    \mathbb P(|T_{N,M}|>t)\leq 4\exp(-\alpha_M t^2)
\end{equation}
where $\alpha_M = \frac{1}{24\delta_M^2}$. So, $\alpha_M\ra \infty$ with $M$, and the bound is independent of $N$. 

We combine 
\eqref{eq:layer-cake} and  \eqref{ineq:com_final_claim} to obtain

 \begin{equation}\label{ineq:just_before_dct}\begin{aligned}
\mathbb{E}\Bigl[\Bigl(\exp(|T_{N,M}|) - 1\Bigr)^2 \Bigr]
&= \int_{0}^{\infty} \varphi'(t)\mathbb{P}\Bigl(|T_{N,M}|>t\Bigr)\, dt\\
&\leq 4\int_0^\infty e^t(e^t-1) \exp(-\alpha_M t^2)\, dt
 \end{aligned}
\end{equation}
which converges to $0$ as $M\ra \infty$ by the dominated convergence theorem. Since the estimate in (\ref{ineq:just_before_dct}) does not depend on the size $N$, we are done. 
\end{proof}

\begin{corollary}[Fluctuations of tracial power series]\label{cor:free-fluctuation}
Suppose $c:\mathbb F_g\to \mathbb C$ is an admissible function, and \[\sum_{w\in\mathbb F_g} |c_w||w|<\infty.\]  Let \[ G_N:=\tr \sum_{w\in\mathbb F_g}c_wU^{(N)w}.\]
Assume in addition that $c_w=\overline{c_{w^{-1}}}$ (so that $G_N$ is real-valued). Then for all real numbers $t$ we have
\begin{equation}\label{eqn:fluctuation-formula}
\lim_{N\to \infty} \left( \log \mathbb E \left[ \exp (tG_N)\right]-t\mathbb E G_N\right) = \left(\frac12 \sum_{w\in\mathbb F_g} |c_w|^2|w|\right) t^2.
\end{equation}
In particular, the sequence $(G_N-\mathbb E G_N)_{N=1}^\infty$ converges in distribution to a normal random variable with mean $0$ and variance $\sigma^2=\sum_{w\in\mathbb F_g} |c_w|^2|w|.$
\end{corollary}
\begin{proof} Since $\mathbb E G_N= Nc_\varnothing$, (\ref{eqn:fluctuation-formula}) is immediate from the previous theorem applied to $tG_N$. The conclusion about convergence to a normal is then a standard application of the method of moments. 
\end{proof}

{\bf Question:} For admissible functions $c:\mathbb F_g\to \mathbb C$, can the hypothesis
\[
\sum_{w\in\mathbb F_g}|c_w||w|<\infty
\]
of Theorem~\ref{thm:a_free_szego} be weakened to 
\[
\sum_{w\in\mathbb F_g} |c_w| + \sum_{w\in\mathbb F_g} |c_w|^2|w| <\infty ?
\]

\section{Uniform Tail Bounds}\label{sec:uniform-tail-bounds}

A fundamental source of difficulty in the free setting is that, because the full and reduced $C^*$ algebras over the free groups (in two or more generators) are distinct, it is possible for a self-adjoint polynomial $p$ to be positive when evaluated on the free Haar unitaries $u$, but not positive when evaluated on unitary matrices. (This is in contrast to the one variable case, where the limiting Haar unitary $u$ is the operator of multiplication by $e^{i\theta}$ on $L^2(\mathbb T)$, and we always have the inclusion of spectra $\sigma(p(U,U^*))\subset \sigma (p(u,u^*))$.) The purpose of this section is to prove the technical theorem below, which says that, for our purposes, the contribution from these ``bad" matrices $U$, for which $p(U,U^*)$ has eigenvalues lying far outside the spectrum of $p(u,u^*)$, is asymptotically negligible. This is the essence of the strong convergence phenomenon, which enters the proof through Parraud's estimate (\ref{eq:parraud-eq}). 

\begin{theorem}[Uniform tail bounds]\label{thm:uniform-bound}
    Let $p\in M_k(\C)\lip x_,x^*\rip$ be a formally self-adjoint nc polynomial. Assume that
    \begin{enumerate}
        \item {[strict positivity]} $p(u,u^*)$ is positive and invertible in $M_k(\C)\otimes C_r^*(\mathbb F_g)$
        \item {[normalization]} $(\frac{1}{k}\tr_k\otimes \tau)(\log p(u,u^*)) = 0$
    \end{enumerate}
    Let $[a,b]\subset (0,\infty)$ be an interval containing the spectrum of $p(u,u^*)$, and let $f:\mathbb R\to \mathbb R$ be a smooth (at least $C^6$) function which agrees with $\log x$ on $[a,b]$. 
    For each integer $N\geq 1$ we consider the random variable
    \[F_N := \tr_{Nk} f(p(U^{(N)},U^{(N)*}))
    \]
    Then there exist constants $c>0, C>0$ (depending only on $p$ and $f$) such that
    \begin{equation}\label{eqn:subgaussian}
    \mathbb{P}\Bigl(|F_N|>t\Bigr)\leq C \exp(-ct^2) 
    \end{equation}
    for all $N$ and all $t>0$.

Moreover if $\mathcal P$ is a bounded family in the sense of Definition~\ref{def:bounded-family}, where each $p\in\mathcal P$ satisfies the hypotheses (1) and (2), and there is a single interval $[a,b]\subset (0,\infty)$ containing the spectrum of $p(u,u^*)$ for all $p\in \mathcal P$, then there are uniform constants $c,C$ such that \eqref{eqn:subgaussian} holds for all $p\in \mathcal P$. 
\end{theorem}

\begin{proof}
For a fixed nc polynomial $p$, there is a uniform norm bound $\|p(U,U^*)\|\leq K$ over all unitary operators $U$, and since $f$ is smooth it is Lipschitz on $[-K,K]$. Lemmas~\ref{eq:monomial-lipschitz-bound} and \ref{lem:lipschitz-on-sa} then imply that $F_N$ is a Lipschitz function on the unitary group with Lipschitz constant $L\sqrt{N}$, where the constant $L$ depends only on the nc polynomial $p$ and is independent of $N$. More generally, for families $\mathcal P$ as in the hypothesis, we will still have a uniform norm bound $K$ over all $p\in \mathcal P$, and since there is a single interval $[a,b]\subset (0, \infty)$ containing all the spectra of the $p(u,u^*)$, we may fix a single smooth $f$ equal to $\log x$ on this $[a,b]$, and $f$ may be chosen to be six times boundedly differentiable on $\mathbb R$.  We also observe that for a family of polynomials of bounded degree with bounded coefficients (and fixed $f$), there is a uniform choice of $L$ valid for the whole family. \medskip

By concentration of measure (Theorem~\ref{thm:concentration-of-measure}) we have for all $N\geq 1$ and all $t>\E F_N$
\begin{align*}
    \mathbb{P}(F_N>t) &= \mathbb{P}(F_N> \mathbb{E}F_N+(t-\mathbb{E}F_N))\\[1.2em]
    &\leq \exp(\frac{-N(t-\E F_N)^2}{12(L\sqrt{N})^2})\\[1.2em]
    &= \exp(\frac{-(t-\E F_N)^2}{12L^2}).
\end{align*} 

Since $-F_N$ Lipschitz with the same constant it follows from the same computation that 
\begin{equation}\label{eqn:first-tail-estimate}
  \mathbb{P}\Bigl(|F_N| >t\Bigr)\leq 2\exp\Bigl(\frac{-(t-|\E F_N|)^2}{12L^2}\Bigr)
\end{equation}
for all $t>|\E F_N|$. 

It follows from Parraud's estimate Theorem~\ref{thm:parraud-asymptotics}, and our normalization $\tau(\log\,p(u,u^*))=0$, that for all $N$ sufficiently large
\begin{equation}
\left| \mathbb E F_N\right|=\left|  \E \left[\tr\,f(p(U^{},U^{*})\right]\right| \leq C(p)\|f\|_{C^6}\frac{\left(\log N\right)^2 k^3}{N}
\end{equation}
All that matters for our purposes is that the right hand side is $O(1)$ as $N\to \infty$, with a constant that depends only on $p$ and the choice of smooth function $f$. Moreover, given a bounded family $\mathcal P$, there is a {\em uniform} choice of constant $C$ valid for all $p\in \mathcal P$ (this follows from the proof of \cite[Lemma 4.4]{Parraud2022}, see remark following \cite[Theorem 4.1]{Parraud2022}). 

In particular we conclude that the expectations $\E F_N$ are uniformly bounded in $N$; let $M=\sup_N |\E F_N|$. By the remark above, this $M$ can be chosen uniformly over sets of polynomials of bounded degree with bounded coefficients. Applying (\ref{eqn:first-tail-estimate}) we obtain for all $N$ 
\begin{equation}\label{eq:parraud-consequence}
    \mathbb{P}\Bigl(|F_N| >t\Bigr) \leq 2\exp\Bigl(\frac{-(t-M)^2}{12L^2}\Bigr) \quad \text{for all } t>M
\end{equation}
Finally, by choosing $C>0$ sufficiently large and $c>0$ sufficiently small (depending only on $M$ and $L$), we may arrange that
\begin{equation*}
   \mathbb{P}\Bigl(|F_N| >t\Bigr) \leq C\exp\Bigl(-ct^2\Bigr) 
\end{equation*}
for all $N$ and all $t>0$. Since $M$ and $L$ can be chosen uniformly over bounded families, the constants $C,c$ can be chosen uniformly as well. 
\end{proof}

\section{Spectral Radius of Free Haar Unitary Pencils}\label{sec:spectral-radius}

The theorem of the last section requires, as one of its inputs, control over the spectrum of a self-adjoint polynomial in free Haar unitaries. In particular we will be applying it, in the next section, to polynomials of the form
\[
p(u,u^*)= L_X(u)L_X(u)^*
\]
where $L_X(u)=I_k\otimes 1 + \sum_{j=1}^g X_j\otimes u_j$. Here $X_1, \dots, X_g$ are fixed $k\times k$ matrices. 
Therefore, our goal in the present section is to compute the spectral radius 
\[\rho\left(\sum_{j=1}^g X_j\otimes u_j\right)\]
only in terms of the coefficient operators $X_1,...,X_g$. In fact we will be able to do this not just for matrix coefficients but for arbitrary bounded operators on Hilbert space $X_j\in B(\mathcal H)$. 

As one expects, this will be done using Gelfand's formula 

\[\rho\left(\sum_{j=1}^g X_j \otimes u_j\right) = \lim_{n\ra \infty}\left\|\left(\sum_{j=1}^g X_j \otimes u_j\right)^n \right\|^{\frac{1}{n}}.\]
We can multiply out the n-fold product of the sum $\sum X_j \otimes u_j$ to obtain
\[\left(\sum_{j=1}^g X_j \otimes u_j\right)^n  = \sum_{|w| = n}X^{w}\otimes u^{w}\]
where recall $X^{w}$ denotes a word in $X_1,...,X_{g}$ of length $|w|$, and $u^{w}$ a word in $u_1,...,u_{g}$. 

So, to compute the spectral radius we need estimates on the quantities 
\begin{equation}\label{eqn:degree-n-sum}
\left\|\sum_{|w| = n}X^{w}\otimes u^{w}\right\|.
\end{equation}

In 1979,  Haagerup (\cite{Haagerup1979}, Lemma 1.4) proved
    \[\left\|\sum_{w\in W_n}\alpha_w u^w\right\| \leq (n+1)(\sum_{w}|\alpha_w|^2)^{\frac{1}{2}}\]
    for any finitely supported $(\alpha_w):\mathbb{F}_g\ra \mathbb{C}$, where $W_n\subset \mathbb{F}_g$ denotes the set of all reduced words of length $n$.

In 1993, Haagerup and Pisier~\cite[Proposition~1.1]{HP1993} replaced $\mathbb{C}$ by an arbitrary $C^*$-algebra $\mathcal{A}$. They showed that
\begin{equation}\label{eq:haagerup-pisier-ineq-lower-bound}
\Bigl\| \sum_{w} \alpha_w \otimes u^w \Bigr\|
\;\geq\;
\max\Bigl\{
\Bigl\| \sum_{w} \alpha_w^* \alpha_w \Bigr\|^{\frac{1}{2}},
\;
\Bigl\| \sum_{w} \alpha_w \alpha_w^* \Bigr\|^{\frac{1}{2}}
\Bigr\}
\end{equation}
for any finitely supported family $(\alpha_w) \colon \mathbb{F}_d \to \mathcal{A}$.

Moreover, if $(\alpha_w)$ is supported on the set of generators and there inverses (i.e. words of length 1)
\[
\{ \gamma_1, \dots, \gamma_g, \gamma_1^{-1}, \dots, \gamma_g^{-1} \},
\]
then
\begin{equation}\label{eq:haagerup-pisier-ineq-upper-bound}
\Bigl\| \sum_{w} \alpha_w \otimes u^w \Bigr\|
\;\leq\;
2
\max\Bigl\{
\Bigl\| \sum_{w} \alpha_w^* \alpha_w \Bigr\|^{\frac{1}{2}},
\;
\Bigl\| \sum_{w} \alpha_w \alpha_w^* \Bigr\|^{\frac{1}{2}}
\Bigr\}.
\end{equation}
  
    So, Haagerup and Pisier's lower bounds~\eqref{eq:haagerup-pisier-ineq-lower-bound} immediately give lower estimates on (\ref{eqn:degree-n-sum}), but not upper estimates. However there is another generalization of Haagerup's inequality due to Buchholz \cite{Buchholz_1999}, which can also be found in \cite[Theorem 9.7.4]{Pisier_2003}, which provides an upper bound sufficient for our purpose.  

To state this inequality we need some notation.     
Let $X:\mathbb{F}_g\to \mathcal B(\mathcal H)$ be a finitely supported function and fix $n\in\mathbb{N}$.
For each $0\le k\le n$, define an operator-valued matrix
\[
X^{[n,k]}
=
\bigl(X^{[n,k]}_{v,w}\bigr)_{v\in W_k,\ w\in W_{n-k}}
\]
by
\[
X^{[n,k]}_{v,w}
:=
X(vw^{-1})\,\mathbf 1_{\{\,|vw^{-1}|=n\,\}} .
\]
(Here as before $W_m$ denotes the set of reduced words of length $m$. The rows of the matrix are indexed by $v\in W_k$, the columns by $w\in W_{n-k}$, and the $(v,w)$ entry is given by the above formula.) Equivalently, $X^{[n,k]}$ defines a bounded operator
\[
X^{[n,k]}:\ell^2(W_{n-k},\mathcal H)\to \ell^2(W_k,\mathcal H)
\]
acting by
\begin{equation}\label{eqn:originial-xnk-def}
(X^{[n,k]}f)(v)
=
\sum_{w\in W_{n-k}}
X(vw^{-1})\mathbf 1_{\{|vw^{-1}|=n\}}\,f(w),
\qquad v\in W_k .
\end{equation}
We let $\| X\|_{[n,k]}$ we denote the norm of this operator $X^{[n,k]}$.
Two special cases occur when $k=0$ and $k=n$. In both cases the length condition $|vw^{-1}|=n$ is automatic, so the matrices reduce to the row and column operators formed by the coefficients of $X$ indexed by words of length $n$.

\begin{theorem}[\cite{Buchholz_1999}]\label{thm:buchholz-haagerup}
 Let $\mathcal H$ be a Hilbert space, and let $X:\mathbb{F}_g\ra \mathcal B(\mathcal H)$ be a finitely supported function. Then, for every $n \in \mathbb N$,
    \[\left\| \sum\limits_{w \in W_n}X(w)\otimes u^w \right\| \leq (n+1)\max_{0\leq k\leq n}\Bigl\{\| X\|_{[n,k]}\Bigr\}\]
\end{theorem}
\begin{remark}
    This recovers the Haagerup-Pisier inequality~\eqref{eq:haagerup-pisier-ineq-upper-bound} in the case $n=1$. 
\end{remark}
    
It will be convenient for us to work with a slightly different operator whose norm will be the same as that of $X^{[n,k]}$. For our purposes we can restrict ourselves to the special case of functions $X: \mathbb F_g\to \mathcal B(\mathcal H)$ supported on the monoid $\mathbb F_g^+$, that is, on words which involve only the generators $\gamma_1, \dots,\gamma_g$ and not their inverses. In this case we redefine $W_k$ to be words of length $k$ in $\mathbb F_g^+$. When using $W_k$ to index a sum we always give it the lexicographic order. We now redefine $X^{[n,k]}$ as     
 \begin{align}\label{new-xnk-def}
(X^{[n,k]}f)(v)&=\sum_{w\in W_{n-k}} X(vw)\mathbf 1_{\{|vw|=n\}}\,f(w)\\
&= \sum_{w\in W_{n-k}} X(vw)f(w),  \qquad v\in W_k 
\end{align}
It is a simple matter to verify that for functions $X$ supported on $\mathbb F_g^+$, the norm of this operator is the same as the norm of the operator defined in (\ref{eqn:originial-xnk-def}).
(Essentially, this amounts to applying the unitary change of variable $f(w)\to f(w^{-1})$ on the domain side.)

We denote the edge cases $X^{[n,0]}$ and $X^{[0,n]}$ by $X_n^{row}$ and $X_n^{col}$ respectively; observe that these operators are represented as block matrices formed by listing the operators $X(w)$ over all $w\in W_n$ in a row or column respectively (again, always in lexicographic order). We apply the above constructions to the special case where the function $X(w)$ is defined to be $X^w$ for a fixed $g$-tuple of operators $X=(X_1, \dots, X_g)$ on a Hilbert space $\mathcal H$. In this instance, the operators $X_k^{col}$ and $X_k^{row}$ are the operators with block matrices formed by placing all the monomials $X^{w}$, $|w|=k$, in a column or row respectively. An inspection of our revised definition of $X^{[n,k]}$ shows that in this special case, for each $k=0, \dots , n$, the operator $X^{[n,k]}$ may be factored as 
\[
X^{[n,k]} = X_k^{col} X_{n-k}^{row}.  
\]
We use the notation $\|X\|_{C^kR^{n-k}}:= \|X^{[n,k]}\|$ in this case. From the above factorization we deduce the inequality
\[
\|X\|^{[n,k]}=\|X\|_{C^kR^{\,n-k}}\leq \|X_k^{col}\| \|X_{n-k}^{row}\|,
\]
for all $0\leq k\leq n$. 
 In particular the resulting norms are the {\it $k$-fold column} and {\it $k$-fold row} norms of the tuple $(X_1, \dots, X_g)$:
\[
\|X\|_{C_k}:= \|X_k^{col}\| = \left\| \sum_{|w|=k} X^{w*}X^{w}\right\|^{1/2}; \quad \|X\|_{R_k}:= \|X_k^{row}\| = \left\| \sum_{|w|=k} X^{w}X^{w*}\right\|^{1/2}.
\]
In this special case, the Buchholz and Haagerup-Pisier inequalities combine to give: 
\begin{proposition}[Buchholz and Haagerup-Pisier inequalities, special case]\label{prop:Buchholz-Haagerup-Pisier-ineq}
For operators $X_1, \dots, X_g\in\mathcal B(\mathcal H)$, and free Haar unitaries $u_1, \dots, u_g$, we have for each integer $n\geq 1$
\begin{equation}\label{eqn:buchholz-redux}
\max \{\|X\|_{C_n}, \|X\|_{R_n} \} \leq \left\|\sum_{|w|=n} X^w \otimes u^w\right\| \leq (n+1)\max_{0\leq k\leq n}\{ \|X\|_{C^k}\|X\|_{R^{n-k}}\}.
\end{equation}
\end{proposition}

Next, for a $g$-tuple of bounded operators $X_1, \dots, X_g$ acting in a Hilbert space $\mathcal H$, we define two (completely positive) linear maps on $\mathcal B(\mathcal H)$
\[
\Phi_{row}^X(T) = \sum_{j=1}^g X_jTX_j^*, \quad \Phi_{col}^X(T) = \sum_{j=1}^g X_j^*TX_j.
\]
Since the maps are completely positive we have $\|\Phi_{row}^X\| = \|\Phi_{row}^X(I)\| = \|\sum_{j=1}^g X_jX_j^*\|=\|X\|_{row}^2$ and similarly $\|\Phi_{col}^X\| = \|X\|_{col}^2$; iterating we find
\[
\| (\Phi_{row}^X)^n\| = \|X\|_{R_n}^2, \quad \| (\Phi_{col}^X)^n\| = \|X\|_{C_n}^2.
\]

We then define the {\em row} and {\em column} spectral radii of $X$ to be
\begin{equation*}\label{def:col-row-spectral-radii}
    \rho_{row}(X) = [\rho(\Phi_{row}^X)]^{1/2}, \quad \rho_{col}(X) = [\rho(\Phi_{col}^X)]^{1/2}.
\end{equation*}
We note that from the Gelfand spectral radius formula we have
\begin{equation}\label{eq:col-row-spectral-radius}
\rho_{row} (X) =\lim_{n\to \infty} \|X\|_{R_n}^{1/n}, \quad \rho_{col} (X) =\lim_{n\to \infty} \|X\|_{C_n}^{1/n}.
\end{equation}

The quantities $\rho_{row}(X)$ and $\rho_{col}(X)$ were studied independently, in various guises, by Bunce \cite{BunceOuterSpec1984}, Popescu \cite{PopescuOuterSpec2014}, Pascoe \cite{pascoe2019outerspectralradiusdynamics}, and Shalit and Shamovich  \cite{ShalitMosheShamovich2025}.

In the case of matrix tuples (that is, when $\mathcal H$ is finite dimensional), the two spectral radii coincide, and are equal to the so-called {\em outer spectral radius} (\cite{pascoe2019outerspectralradiusdynamics}): 
\[
\rho_{out} (X):= \rho \left(\sum_{j=1}^g \overline{X_j}\otimes X_j\right)^{1/2}.
\]
For the reader's convenience we give a short proof of this fact. 
\begin{lemma}\label{lem:radius-haar-pencils}
Let $X_1,...,X_d\in M_k(\mathbb{C})$. Then  \[\rho_{col}(X)=\rho_{row}(X)= \rho_{out}(X).\]
\end{lemma}
\begin{proof}
Consider $\Phi^X_{\mathrm{row}}$ and $\Phi^X_{\mathrm{col}}$ as operators on
$M_k(\mathbb{C})$ equipped with the Hilbert–Schmidt inner product
\[
\langle A,B\rangle=\tr(B^*A).
\]
For all $A,B\in M_k(\mathbb{C})$ we compute
\[
\begin{aligned}
\langle \Phi^X_{\mathrm{row}}(A),B\rangle
&=\tr\Bigl(B^*\sum_i X_i A X_i^*\Bigr)
= \sum_i \tr\bigl(X_i^* B^* X_i A\bigr) \\
&=\tr\Bigl(\Bigl(\sum_i X_i^* B X_i\Bigr)^* A\Bigr)
= \langle A,\Phi^X_{\mathrm{col}}(B)\rangle,
\end{aligned}
\]
by the cyclicity of the trace. Thus
\[
(\Phi^X_{\mathrm{row}})^*=\Phi^X_{\mathrm{col}}.
\]
Since spectral radius is preserved under taking adjoints, it follows that
\[
\rho(\Phi^X_{\mathrm{row}})
=\rho\bigl((\Phi^X_{\mathrm{row}})^*\bigr)
=\rho(\Phi^X_{\mathrm{col}}).
\]
And the lemma follows from Pascoe's \cite{pascoe2019outerspectralradiusdynamics} computation of the outer spectral radius. (This comes down to checking that $\sum_{j=1}^g \overline{X_j}\otimes X_j$ is the matrix of $\Phi_{row}^X$ in a suitable basis.)
\end{proof}

\begin{lemma}\label{lem:roc-lemma}
    Let $a_n,b_n\geq 0$ be sequences of real numbers, and suppose that 
    \[\limsup a_n^{1/n}=\alpha, \quad \limsup b_n^{1/n}=\beta.\] 
    Define
    \[
    c_n:=\max\{a_k b_{n-k} : 0\leq k\leq n\}. 
    \]
    Then \[\limsup c_n^{1/n} \leq \max\{\alpha,\beta\}.\]. 
\end{lemma}
\begin{proof}
    Without loss of generality we may assume $\beta\leq \alpha<\infty$. Consider the power series $a(x)=\sum_{n=1}^\infty a_nx^n$, $b(x)=\sum_{n=1}^\infty b_nx^n$. By hypothesis the series have positive radii of convergence $\alpha^{-1}, \beta^{-1}$ respectively. Since the coefficients are positive, the series are absolutely convergent within these radii, and their Cauchy product 
    \[
    c(x)=a(x)b(x) = \sum_{n=0}^\infty \left( \sum_{k=0}^n a_kb_{n-k}\right) x^n
    \]
    has radius of convergence at least $\alpha^{-1}$. This means that 
    \[
    \limsup c_n^{1/n}\leq\limsup \left(\sum_{k=0}^n a_kb_{n-k}\right)^{1/n} \leq \alpha. 
    \]
\end{proof}

We are now ready to prove the main theorem of this section:
 \begin{theorem}\label{thm:spectral-radius}
        
If  $u_1, \dots, u_g$ are free Haar unitaries and $X_1, \dots, X_g$ is a $g$-tuple of bounded operators acting in a Hilbert space $\mathcal H$, then  \[\rho\left(\sum\limits_{i=1}^g X_j\otimes u_j\right) = \max\Bigl\{\rho_{row}(X), \;\rho_{col}(X)\Bigr\}.\]
In particular, if the $X_j$ are matrices, 
 \[\rho\left(\sum\limits_{j=1}^g X_j\otimes u_j\right) = \rho\left(\sum_{j=1}^g \overline{X_j} \otimes X_j\right)^{1/2}.\]
\end{theorem}

\begin{proof}
The lower bound 
\[
 \rho\left(\sum\limits_{j=1}^g X_j\otimes u_j\right) \geq \max\Bigl\{\rho_{row}(X), \;\rho_{col}(X)\Bigr\}
\]
follows from taking $n^{th}$ roots and limits in the Haagerup--Pisier
lower bound~\eqref{eqn:buchholz-redux}, together with \eqref{eq:col-row-spectral-radius}.

Similarly, the corresponding upper bound follows from the upper bound in \eqref{eqn:buchholz-redux} and Lemma~\ref{lem:roc-lemma} applied with $a_n=\|X\|_{C_n}, b_n=\|X\|_{R_n}$ (using \eqref{eq:col-row-spectral-radius} again). 
\end{proof}

\begin{remark}
Similar upper bounds can be obtained for other free random variables besides the free Haar unitaries.
In particular, one may replace the upper bound from \eqref{prop:Buchholz-Haagerup-Pisier-ineq} by the version of the Haagerup inequality due to de~la~Salle \cite{DeLaSalle2009}; these estimates then apply for example to the so-called ``free circular" elements.
\end{remark}

\section{The free SSLT for linear pencils}\label{sec:free-sslt-pencil}

For $g$-tuples of matrices $X=(X_1, \dots, X_g)$, we recall the operator space norms
\[
\|X\|_{\max \ell^1} := \sup_U \left\|\sum_{j=1}^g X_j\otimes U_j\right\|, 
\]
where the supremum is taken over all systems of unitary matrices $U_1, \dots, U_g$, and
\[
\|X\|_{row}:=\left\| \sum_{j=1}^g X_jX_j^*\right\|^{1/2}.
\]
 It is well known that $\|X\|_{\max \ell^1}$ is equal to the norm of $\sum_{j=1}^g X_j\otimes \mathfrak u_j$ in $M_k(\mathbb C)\otimes C^*(\mathbb F_g)$,  where the $\mathfrak u_j$ are the canonical generators of the {\em full} group $C^*$-algebra $C^*(\mathbb F_g)$.
\begin{proposition}[Uniform estimates on moments of determinants]\label{prop:uniform-bound-l2}
    Fix $g\geq 1$ and let  \( U_1^{(N)}, \ldots, U_g^{(N)} \) be independent Haar distributed random unitaries. 
\begin{itemize}    
 \item[(1)] {[max $\ell^1$ coefficients, all real exponents]}   Let $0\leq r<1$, $k\in \mathbb{N}$, $m\in \mathbb R$. Then there exists a constant $C_1(k,m,r)$ such that
 \[\sup\limits_{N\in \mathbb{N}}\mathbb{E}\left[\det |L_X(U^{(N)})|^{2m}\right]\leq C_1(k,m,r)\]
 for all \( X = (X_1, \ldots, X_g) \in M_{k}(\mathbb{C})^g \) such that $\|X\|_{\max\ell^1}\leq r$.

\item[(2)] {[row ball coefficients, positive exponents]} Let $0\leq r<1$, $k\in \mathbb{N}$, $m\geq 0$. Then there exists a constant $C_2(k,m,r)$ such that
 \[\sup\limits_{N\in \mathbb{N}}\mathbb{E}\left[\det |L_X(U^{(N)})|^{2m}\right]\leq C_2(k,m,r)\]
 for all \( X = (X_1, \ldots, X_g) \in M_{k}(\mathbb{C})^g \) such that $\|X\|_{row}\leq r$.

\end{itemize}
 
\end{proposition}

\begin{proof}[Proof of Proposition~\ref{prop:uniform-bound-l2}]
 (1): By definition the assumption $\|X\|_{\max \ell^1} \leq r<1$ means that $\|\sum_{j=1}^g X_j\otimes U_j\|\leq r<1$ for all unitaries $U$. We may then set 
\begin{equation}\label{eqn:trace-log-single}
	f_N^{(X)}(U) := \tr\log\left( I_k \otimes I_N + \sum_{j=1}^{g} X_j \otimes U_j \right)
\end{equation}
and put  $F_N^X = f_N^X + \overline{f_N^X}$. From Lemma~\ref{lem:log-lipschitz}, the function $T\to \tr \log (I+T)$ is Lipschitz with constant at most $C_r\sqrt{N}$ on the set of $N\times N$ matrices of norm at most $r<1$. We conclude that the random variables $F_N^X$ are Lipschitz on $\mathcal U(N)^g$ with a constant $L\sqrt{N}$, where $L$ is uniform in $N$ and $\|X\|_{\max\ell^1}\leq r$. Applying concentration of measure (Corollary~\ref{cor:com-cor}), the $F_N^X$ obey a uniform sub-Gaussian tail bound. Since $|\det L_X(U^{(N)})|^{2m} =\exp\left( m F_N^X\right)$ for all $m\in\mathbb R$, the conclusion follows. Indeed, 
 \begin{align*}
     \mathbb E|\det L_X(U^{(N)})|^{2m} &\leq \mathbb E[\exp(mF_N^X)]\\
     &\leq \int_0^\infty e^{t}\,  \mathbb P\left(|F_N^X|>\frac{t}{m}\right)\, dt \\
&\leq C\int_0^\infty e^{t}e^{-ct^2/m^2}\, dt<\infty,
 \end{align*}
where the last bound is independent of $N$.

(2): For fixed $0\leq r<1$, we consider the set of matrix polynomials in $M_{k}(\C)\otimes C^*_r(\mathbb{F}_g)$ 
    \[\mathcal P =\{p_X(u,u^*) = L_X(u)L_X(u)^*: X_1, \dots, X_g\in M_k(\C), \|X\|_{row}\leq r\}.\]

This is clearly a bounded family in the sense of Definition~\ref{def:bounded-family}. We will verify the hypotheses of Theorem~\ref{thm:uniform-bound} for this $\mathcal P$. For each $X$, the polynomial $p_X(u,u^*)$ is positive by construction. By Theorem~\ref{thm:spectral-radius}, the pencil of free Haar unitaries $\sum_{j=1}^g X_j\otimes u_j$ has spectral radius at most $r$. It follows that $p_X(u,u^*)$ is invertible, and moreover we can choose a single interval $[a,b]\subset (0,\infty)$ such that the spectrum $\sigma(p_X(u,u^*))\subset [a,b]$ for all $\|X\|_{row}\leq r$. Indeed, observe that the functions $X\to \|L_X(u)\|$, $X\to \|L_X(u)^{-1}\|$ are continuous, and thus attain maxima, on the compact set of $X$ with $\|X\|_{row}\leq r$. We may then take
\[
a= \left(\max_{\|X\|_{row}\leq r} \|L_X(u)^{-1}\|^2 \right)^{-1}, \quad b=\max_{\|X\|_{row}\leq r} \|L_X(u)\|^2.
\]

 Next, observe that for each integer $n\geq 1$ we have
 \begin{align*}(\tr_k\otimes \tau)\left(\sum_{i=1}^g X_i\otimes u_i\right)^n&=\sum_{|w|=n} Tr_k(X^w)\tau(u^w)\\
   &=0,
   \end{align*}
 and since the spectral radius of $\sum_{i=1}^g X_j\otimes u_j$ is strictly less than $1$, it follows from properties of the Fuglede-Kadison determinant (Lemma~\ref{lem:FK-determinant}(2)) that $\Delta_{FK}(L_X(u)) =1$. Therefore by Lemma~\ref{lem:FK-determinant}(1)
 \begin{align*}
   (Tr_k\otimes \tau)\log p_X(u,u^*) &=2\log \Delta L_X(u)\\
   &=0
 \end{align*}
for all $\|X\|_{row}\leq r$. 
 
 Observe that $\|p_X(U,U^*)\|=\|L_X(U)L_X(U)^*\|\leq 4$ for all unitary matrices $U$. Now let $f:\mathbb R\to \mathbb R$ be a smooth (at least $C^6$) function such that $f$ and its derivatives are bounded on $\mathbb R$, and such that $f(x)=\log x$ on $[a,b]$ and  $f(x)\geq \log x$ for all $0<x\leq 4$. Put
 \[
 F_N^X := \tr\, f(p_X(U^{}, U^*))
 \]
 
We may now apply Theorem~\ref{thm:uniform-bound} to the bounded family $\mathcal P$, so there are constants $C,c>0$ such that for all systems of $k\times k$ matrices $X=(X_1, \dots, X_g)$ with $\|X\|_{row}\leq r$, and all $N$,
 \begin{equation}\label{eqn:tail-bound-corollary}
   \mathbb P(|F_N^X|>t) \leq C\exp(-ct^2).
 \end{equation}

 For a positive semidefinite matrix $A$ with a $0$ eigenvalue we interpret $\tr \log A:=-\infty$ and $\exp(-\infty):=0$. With these conventions we have $\det A =\exp \tr \log A$ for all positive semidefinite $A$. With the smooth function $f$ as above, we then have for all $\|X\|_{row}\leq r$ and all $N$
 \begin{equation} \tr_{Nk} \log [L_X(U^{(N)})L_X(U^{(N)})^* ] \leq F_N^X
\end{equation}
(pointwise as functions of the $U^{(N)}$). Multiplying by $m\geq 0$, exponentiating, and taking expectations, we obtain
 \begin{equation}
   \mathbb E|\det L_X(U^{(N)})|^{2m} \leq \mathbb E[\exp(mF_N^X)]
   \end{equation}
 Since we have the sub-Gaussian tail estimate (\ref{eqn:tail-bound-corollary}), uniform in $N\geq 1$ and $\|X\|_{row}\leq r$, we finally obtain the bound as in the proof of part (1)
 \begin{align*}
     \mathbb E|\det L_X(U^{(N)})|^{2m} &\leq \mathbb E[\exp(mF_N^X)]\\
     &\leq \int_0^\infty e^{t}\,  \mathbb P\left(|F_N^X|>\frac{t}{m}\right)\, dt \\
&\leq C\int_0^\infty e^{t}e^{-ct^2/m^2}\, dt<\infty. 
 \end{align*}
This completes the proof.
\end{proof}

We are now ready to prove the more general form of Theorem~\ref{thm:intro-free-pencil} from the introduction, which allows for quotients of determinants. (Indeed, Theorem~\ref{thm:intro-free-pencil} is just the case $C=D=0$ in Theorem~\ref{thm:free-sslt-pencil} below.) It is important to observe in the statement that we can allow for a weaker spectral radius assumption in the numerator, than in the denominator. We will comment on this after the proof. The proof depends on a couple of technical statements regarding the $\max$ spectral radius $\rho_{\max}(X)$; for clarity of exposition we defer the proofs of these to Section~\ref{sec:rational-prelim}. By definition, $\rho_{\max}(X)$ is the spectral radius of the element $\sum_{j=1}^g X_j\otimes \mathfrak{u}_j$, (in contrast to the free Haar unitaries $u_j$, which generate the {\em reduced} group $C^*$-algebra $C^*_r(\mathbb F_g)$).

\begin{theorem}[Free SSLT for quotients of monic pencils]\label{thm:free-sslt-pencil}
Let $A, B, C, D$ be $g$-tuples of square matrices with 
\[
\rho_{out}(A), \rho_{out}(B)<1, \quad \text{and}\quad \rho_{\max}(C), \rho_{\max}(D)<1.
\]
Then
   \begin{equation}\label{eqn:pencil-SSLT-main}
   \lim_{N\to \infty} \mathbb E \left[ \frac{\det L_A(U^{(N)})}{\det L_C(U^{(N)})}\frac{\det L_B(U^{(N)})^*}{\det L_D(U^{(N)})^*}\right] = \frac{\det L_A(\overline{-D})\det L_C(\overline{-B})}{\det L_A(\overline{-B})\det L_C(\overline{-D})}.
   \end{equation}
\end{theorem}
\begin{proof}
We first fix $r<\frac1g$ and make the stronger assumption that for each of $X=A,B,C,D$
  \begin{equation}\label{eqn:small-abcd}
  \|X_j\|\leq r \quad \text{for each }j=1, \dots, g.
  \end{equation}
  This implies that if $X$ has size $k\times k$, then $|\tr X^w|\leq kr^{|w|}$, and therefore
  \begin{equation}\label{eqn:small-trace-estimate}
  \sum_{|w|=n} |\tr X^w| \leq k(gr)^n,      
  \end{equation}
which is summable in $n$.  Form the sequence of random variables
  \[
  G_N= \sum_{w\in \mathbb F_g^+, w\neq\varnothing} (-1)^{|w|}\left( \frac{\tr A^w}{|w|}-  \frac{\tr C^w}{|w|} \right)\tr \, U^{(N)w} + (-1)^{|w|}\overline{\left( \frac{\tr B^w}{|w|}-  \frac{\tr D^w}{|w|} \right)} \tr \, U^{(N)w*}.
    \]
    Observe that for each $N$, the series is uniformly convergent in the norm of $M_N(\mathbb C)$ and 
    \begin{align*}
      G_N = \tr  & \left[\log L_A(U^{(N)}) - \log L_C(U^{(N)})+\overline{\log L_B(U^{(N)}) - \log L_D( U^{(N)})}\right]
    \end{align*}
    so that the sequence in the left hand side of (\ref{eqn:pencil-SSLT-main}) is $\mathbb E[\exp G_N]$. 

    Using (\ref{eqn:small-trace-estimate}), we can then apply Theorem~\ref{thm:a_free_szego} with
    \[
    c_w =  (-1)^{|w|} \left(\frac{\tr A^w}{|w|}-  \frac{\tr C^w}{|w|}\right), \quad w\in \mathbb F_g^+,
    \]
    \[
    c_{w^{-1}} = (-1)^{|w|} \overline{\left( \frac{\tr B^w}{|w|}-  \frac{\tr D^w}{|w|} \right)}, \quad w\in \mathbb F_g^+,
    \]
    
    to conclude that the limit in (\ref{eqn:pencil-SSLT-main}) exists and has the value
    \begin{align*}
      \exp &\left( \sum_{w\in \mathbb F_g^+} c_w c_{w^{-1}} |w|\right) = \exp\left(\sum_{w\in \mathbb F_g^+} \frac{(\tr A^w -\tr C^w)\overline{(\tr B^w -\tr D^w)}}{|w|}\right)\\
        &= \exp\left(\sum_{n=1}^\infty \frac{ \tr (\sum A_j\otimes\overline{B_j})^n}{n} -\frac{ \tr (\sum A_j\otimes\overline{D_j})^n}{n}-\frac{ \tr (\sum C_j\otimes\overline{B_j})^n}{n} + \frac{ \tr (\sum C_j\otimes\overline{D_j})^n}{n}\right)\\
        &= \exp \left( \tr \left[\log L_A(-\overline{B}) -\log L_A(-\overline{D})-\log L_C(-\overline{B})+\log L_C(-\overline{D})\right]\right) 
    \end{align*}
   which gives the claimed determinantal formula under the assumption (\ref{eqn:small-abcd}). (To verify this, observe that if $Z$ is a matrix with $\|Z\|<1$, then with $\log (I+Z)$ defined by the usual power series, the identity $\det (I+Z) = \exp \tr\, \log(I+Z)$ is valid.)

    We now relax the assumptions on $A,B,C,D$. The theorem's hypothesis on the denominator matrices $C,D$ implies that there exists $r<1$ such that, up to a similarity, we have 
    \[
    \|C\|_{\max\ell^1}\leq r, \quad \|D\|_{\max\ell^1}\leq r.
    \]
     Similarly, up to similarity the $A,B$ will satisfy
    \[
    \|A\|_{row}\leq r, \quad \|B\|_{row}\leq r.
    \]
    (For both of these facts see Proposition~\ref{prop:max-radius} below.) By Proposition~\ref{prop:uniform-bound-l2} we then have that all of the expectations
    \begin{equation}\label{eqn:main-holder-step}
    \mathbb E |\det L_A(U^{(N)})|^4, \quad  \mathbb E |\det L_B(U^{(N)})|^4, \quad  \mathbb E |\det L_C(U^{(N)})|^{-4}, \quad \mathbb E |\det L_D(U^{(N)})|^{-4}
    \end{equation}
    are uniformly bounded in $N$. 
    
    Now, fix $B$ and $D$ satisfying these hypotheses. For each $N$, the term in the left hand side of (\ref{eqn:pencil-SSLT-main}) is a rational function $r_N(A,C)$ of the entries of all of the $A$ and $C$ matrices, regular in the product of the row ball and the $\max \ell^1$ ball. Applying the four-term H\"older's inequality using the estimates (\ref{eqn:main-holder-step}), we conclude that this sequence of rational functions remains uniformly bounded on compact sets of $A$ and $C$ in the row ball and $\max \ell^1$ ball respectively. By Montel's theorem, this sequence of rational functions will converge uniformly on compact sets to a holomorphic function $f_{B,D}(A,C)$.  By analytic continuation this function must agree with the function in the right hand side of (\ref{eqn:pencil-SSLT-main}). 
    \end{proof}

\begin{remark} One cannot expect to improve the stricter size restrictions on the denominator pencils $C$ and $D$. In particular, when $g>1$ the uniform bounds in item (2) of Proposition~\ref{prop:uniform-bound-l2} cannot be extended to negative exponents $m$. For example, if the statement were true for $k=1$, and $m=-1$, by expanding the determinant in terms of trace polynomials and using similar arguments as in the proof of \cite[Proposition 3.6]{jury2025determinantsrandomunitarypencils}, one could prove that for all $N$ the function
\[
(x,\overline{y})\to \mathbb E[L_x(U)^{-1} \overline{L_y(U)^{-1}}]
\]
would be analytic for $x, \overline{y}$ in the Euclidean ball $\mathbb B^g\subset \mathbb C^g$. However, one can compute an explicit power series expansion for this function valid for small $x$ and $y$ and show that it cannot converge in any polydisk of radius greater than $\frac1g$. (If the function were analytic for all $x,y$ in the ball then the series would have to converge in the polydisk $\frac{1}{\sqrt{g}}\mathbb D^g\subset \mathbb B^g$.)

The proof of this requires some rather lengthy computations which we do not include here, but we sketch the key points. If we assume that $\|x\|_1<1$, then $\|\sum_{j=1}^g x_jU_j\|<1$, and one can then show (by computations very similar to those in the proof of \cite[Theorem 4.1]{jury2025determinantsrandomunitarypencils}) that for $x,y\in \mathbb C^g$ with $\|x\|_1, \|y\|_1<1$, 
\begin{equation}\label{eqn:main-integral-inverse}
	 \mathbb E_N[L_x(U)^{-1} \overline{L_y(U)^{-1}}]= \sum_{n = 0}^\infty \sum_{ |\alpha|=n } \hat{c}_{\alpha}(N) \binom{|\alpha|}{\alpha}x^\alpha \overline{y}^\alpha
	\end{equation}
    
where
	\begin{equation*}\label{eqn:main-coefficient-inverse}
		\hat{c}_{\alpha}(N) = \binom{|\alpha|}{\alpha} \binom{N+|\alpha|-1}{|\alpha|} \prod_{j=1}^{g} \binom{N+\alpha_j-1}{\alpha_j}^{-1}.
	\end{equation*}
It is then not hard to show, using Stirling's formula, that for each fixed $N$
\begin{equation}\label{eqn:root-test}
	\limsup_{n \to \infty} \left( \sum_{ |\alpha|=n } \hat{c}_{\alpha}(N) \binom{|\alpha|}{\alpha} \right)^{\frac{1}{n}} \geq  g^2.
\end{equation}
Thus, when we restrict to the diagonals $x=(s,s, \cdots, s), y=(t,t, \cdots, t)$, we obtain
\[
 \mathbb E_N[L_x(U)^{-1} \overline{L_y(U)^{-1}}] = \sum_{n=0}^\infty \left(  \sum_{ |\alpha|=n } \hat{c}_{\alpha}(N) \binom{|\alpha|}{\alpha}\right)(s\overline{t})^n
\] 
where the power series on the right diverges for $|st|>\frac{1}{g^2}$. It follows that the original series \eqref{eqn:main-integral-inverse} cannot converge in a polydisk of radius greater than $\frac1g$. 

This example further shows that there cannot be any $\delta>0$ so that Proposition~\ref{prop:uniform-bound-l2}(2) is valid for $-\delta<m<0$, since by taking direct sums of $k$ copies of scalar $x$, the above counterexample gives a counterexample with $k\times k$ coefficients and $m=-\frac1k$. 
\end{remark}

\section{Stability and regularity for linear pencils and nc rational functions}\label{sec:rational-prelim}

We have already encountered the joint spectral radii $\rho_{max}$ and $\rho_{out}$, but these are actually just two members of a much larger family. In \cite[Definition 2.1 and 2.2]{ShalitMosheShamovich2025} it is shown how, given an operator space $\mathcal{E}$, and $X\in M_k(\mathcal{E})$, one can define a joint spectral radius $\rho_{\mathcal{E}}(X)$ adapted to $\mathcal E$. We will only work with finite dimensional operator spaces so  $\rho_{\mathcal{E}}(X) = \rho_Q(X)$ (in the notation of \cite{ShalitMosheShamovich2025}). In the case when $\mathcal{E}$ is the row operator space we have $\rho_{\mathcal{E}} = \rho_{out}$. The $\rho_{\max}$ we have already introduced turns out to be equal to the spectral radius associated to the operator space $\mathcal E=\max \ell^1$ over $\mathbb C^g$:

\begin{lemma}\label{lem:max-radius}
    For a tuple $X=(X_1, \dots, X_g)$ of $k\times k$ matrices, the operator space spectral radius $\rho_{\max\ell^1}(X)$ is equal to the spectral radius of the operator $\sum_{j=1}^g X_j\otimes \mathfrak u_j$ in the $C^*$-algebra, $M_k(\mathbb C)\otimes C^*(\mathbb F_g)$. 
\end{lemma}
\begin{proof} 
This follows from the definition of $\rho_{\max\ell^1}$ from \cite{ShalitMosheShamovich2025} and the fact that the universal operator algebra over $\max \ell^1_g$ is the full group $C^*$-algebra $C^*(\mathbb F_g)$; for this latter fact see \cite[Theorem 8.12]{Pisier_2003}.
\end{proof}
The following theorem is from \cite[Corollary 2.12]{ShalitMosheShamovich2025}:

\begin{theorem} \label{shalit_shamovich_cor_2_12}

Let $\cE$ be a finite-dimensional operator space, and let $X \in M_n(\cE)$. 
Then $\rho_\cE(X) < 1$ if and only if there exists $S \in \GL_n $ such that $\|S^{-1}XS\|<1$.
\end{theorem}

\begin{proposition}\label{prop:max-radius}
Let $X=(X_1, \dots, X_g)$ be a $g$-tuple of $k\times k$ matrices. Then
\begin{itemize}
        \item[1)] $\rho_{\max}(X)<1$ if and only if there exists $S\in GL_k(\mathbb C)$ and a $g$-tuple $Y$ with $\|Y\|_{\max \ell^1_g}<1$ such that $X=S^{-1}YS$.
        \item[2)] $\rho_{out}(X)<1$ if and only if there exists $S\in GL_k(\mathbb C)$ and a $g$-tuple $Y$ with $\|Y\|_{row}<1$ such that $X=S^{-1}YS$.
    \end{itemize}
\end{proposition}
\begin{proof} Item (1) is immediate from Lemma~\ref{lem:max-radius} and Theorem~\ref{shalit_shamovich_cor_2_12}. Item (2) can also be obtained as an application of Theorem~\ref{shalit_shamovich_cor_2_12}, but is older, see \cite[Theorems 1.1 and 1.9]{pascoe2019outerspectralradiusdynamics} for self-contained proofs and some history.
\end{proof}

We say an nc polynomial $p \in M_d(\mathbb{C})\langle x_1, \cdots, x_g,x_1^*,\cdots,x_g^* \rangle$ 
is \emph{strictly stable} with respect to $\min(\ell^{\infty}_g)$ (the minimal operator space structure on $\mathbb{C}^g$ equipped with $\ell^{\infty}$ norm)
if

\[
\det(p(Z)) \neq 0,  \quad \text{for all }\, Z
\text{ such that } \|Z\|_{\min(\ell^{\infty}_g)} \le 1
\]

where $Z \in \bigcup\limits_{k=1}^{\infty} M_k(\mathbb{C})^g$.  We observe that $ \|Z\|_{\min(\ell^{\infty}_g)}$ is equal to $\max\{\|Z_1\|, \dots, \|Z_g\|\}$. Hence we refer to the set $\{\|Z\|_{\min\ell^\infty_g}<R\}$ as the {\em nc polydisk} of radius $R$.

As a consequence of results in \cite{ShalitMosheShamovich2025} we have a version of \cite[Proposition 7.1]{jury2025determinantsrandomunitarypencils} where the outer spectral radius is replaced by the spectral radius associated to the $\max(\ell^1_g)$ operator space structure.
\begin{proposition}\label{prop:stable-polynomials-generalization}
Suppose $p \in M_d(\mathbb{C})\langle x_1, \cdots, x_g,x_1^*,\cdots,x_g^* \rangle$,
$p(0) = I_d$.  
Then $p$ is strictly stable with respect to $\min(\ell^{\infty}_g)$ if and only if
there exists a tuple $X = (X_1, \ldots, X_g) \in M_k(\mathbb{C})^g$ such that
\[
\rho_{\max \ell^1}(X) < 1
\]
and
\begin{equation}\label{strictly_stable_char_eq}
\det(p(Z)) = \det\!\left(I_k\otimes I_d - \sum_i X_i \otimes Z_i \right)
\quad \text{for all } \|Z\|_{\min(\ell_g^{\infty})} \le 1.    
\end{equation}

\end{proposition}

\begin{proof}    
    That the determinantal representation implies strict stability is immediate from the operator space duality between $\min(\ell^\infty_g)$ and $\max(\ell^1_g)$ and Proposition~\ref{prop:max-radius}(1).
    
    For the converse direction, we first note that by a routine compactness argument, if $p$ is strictly stable with respect to $\min(\ell^\infty_g)$ then there exists an $R>1$ such that for all $Z$ with $\|Z\|_{\min(\ell^1_g)}<R$, we have $\det p(Z)\neq 0$. It follows that the free rational function $r(Z):=p(Z)^{-1}$ is regular in the $\min(\ell^\infty_g)$ ball of radius $R$. The proof then proceeds identically to the proof of \cite[Proposition 7.1]{jury2025determinantsrandomunitarypencils}, using \cite[Theorem 3.4]{ShalitMosheShamovich2025} in place of \cite[Theorem A]{jury-martin-shamovich}.
\end{proof}
\begin{remark}
    The above proof can obviously be generalized to replace $\min\ell^1_g$ and $\max\ell^\infty_g$ with any pair $\mathcal E, \mathcal E^*$ of dual, finite-dimensional operator spaces. 
\end{remark}

\begin{proposition}[Linear fractional determinantal representations]\label{prop:linfrac}
  Let $r(x)$ be an nc rational function of $g$ variables  which is regular at $0$ and normalized to have $r(0)=1$. Then there exist monic linear pencils
  \[
  L_P(x) = I+\sum_{j=1}^g P_j x_j, \quad L_Q(x) = I+\sum_{j=1}^g Q_j x_j
  \]
  such that for all $X\in dom(r)$
  \[
  \det L_P(X) = \det L_Q(X) \det r(X).
  \]
  Moreover $Q$ may be chosen so that $\det L_Q(X)\neq 0$ for all $X\in dom(r)$, and hence we have for all $X\in dom(r)$, $\det L_P(X)=0$ if and only if $\det r(X)=0$. 
\end{proposition}
\begin{proof}
  We fix a {\em minimal realization} $r(x) = v^*(I+\sum Q_jx_j)^{-1} u$. For brevity write $Qx=\sum_{j=1}^g Q_j x_j$. From \cite[Prop. 7.1.5]{cohn-2006} we have the identity
  \[
  \begin{bmatrix} 1 & 0 & 1 \\ u & I+Qx & 0 \\ 0 & v^* & 1\end{bmatrix} =   \begin{bmatrix} 1 & 0 & 0 \\ u & I+Qx & 0 \\ 0 & v^* & 1\end{bmatrix}   \begin{bmatrix} 1 & 0 & 0 \\ 0 & 1 & 0 \\ 0 & 0 & r(x) \end{bmatrix}   \begin{bmatrix} 1 & 0 & 1 \\ 0 & 1 & -(I+Qx)^{-1}u\\ 0 & 0 & 1\end{bmatrix}
        \]
which we write in abbreviated form as $\beta(x) = \alpha(x) \rho(x) \gamma(x)$. Multiplying on the left by $\alpha(0)^{-1}$ and on the right by $\gamma(0)^{-1}$, the identity is rewritten in the form $\tilde{\beta}(x) = \tilde{\alpha}(x) \rho(x) \tilde{\gamma}(x)$, where now each factor is equal to $1$ at $x=0$. In particular $\tilde{\alpha}$ and $\tilde{\beta}$ are monic linear pencils, and $\det \tilde{\alpha}(X) = \det L_Q(X)$. Since the realization is minimal, it follows that $\det L_Q(X)=0$ if and only if $X\notin dom(r)$. We finally observe that $\det \tilde{\gamma}(X)\equiv 1$, and there is a monic pencil $L_P$ so that $\det \tilde{\beta}(X) = \det L_P(X)$. Since by construction $\det \rho(x) = \det r(x)$, these $P$ and $Q$ satisfy the conclusion of the theorem.         
\end{proof}

\begin{corollary}\label{cor:free-sslt-rational} Let $q(x), r(x)$ be nc rational functions of $g$ variables whose domains contain an nc polydisk of radius $R>1$ centered at $0$, and which are strictly stable with respect to the row ball. 
 Then there exist pencils $A,B$ strictly stable for the row ball, and pencils $C,D$ strictly stable for $\min\ell^\infty_g$, such that 
 \[
   \det q(X) = \frac{\det L_A(X)}{\det L_C(X)}, \quad    \det r(X) = \frac{\det L_B(X)}{\det L_D(X)}
   \]
   Moreover
   \[
   \lim_{N\to \infty} \mathbb E [ \det q(U^{(N)}) \overline{ \det r(U^{(N)})}] = \frac{\det L_A(\overline{-D})\det L_C(\overline{-B})}{\det L_A(\overline{-B})\det L_C(\overline{-D})}.
   \]
\end{corollary}
\begin{proof}
    The first claim is an immediate consequence of Proposition~\ref{prop:linfrac}. The second claim then follows by Theorem~\ref{thm:free-sslt-pencil}.
\end{proof}

\printbibliography
\end{document}